\documentclass[11pt,a4paper, reqno]{amsart}
\usepackage{amssymb, paralist}
\usepackage{amsfonts, amsmath, wasysym}
\usepackage{latexsym}
\usepackage{graphicx, color}
\usepackage{indentfirst}
\usepackage{stmaryrd}

\usepackage{amssymb}
\usepackage{amsmath}
\usepackage{graphicx, color}
\usepackage{indentfirst}

\newcommand{\C}{\mathbb{C}}

\newcommand{\N}{\mathbb{N}} 

\newcommand{\R}{\mathbb{R}}
\newcommand{\Z}{\mathbb{Z}}
\newcommand{\tg}{\tilde{g}} 
\newcommand{\Sph}{\mathbb{S}}

\newcommand{\re}{\mathbb R}
\newcommand{\ce}{\mathbb C}

\newcommand{\<}{\left<}

\newcommand{\ds}{\displaystyle}
\renewcommand{\>}{\right>}
\newcommand{\flecha}{\longrightarrow}

\newcommand{\gorro}{\widetilde}

\newcommand{\mc}[1]{\mathcal{#1}}

\newcommand{\Rm}{{\rm R}}

\def\vle{{ vol}}

\def\dim{{\rm dim}}

\def\fle{\rightarrow}
\def\parcial#1#2{\fracc{\partial #1}{\partial#2}}

\def\({\left (}
\def \){\right)}
\newtheorem{mainthm}{Theorem}[]
\newtheorem{maincor}[mainthm]{Corollary}
\newtheorem{thm}{Theorem}[section]

\newtheorem*{thm*}{Theorem} 
\newtheorem{cor}[thm]{Corollary}\newtheorem{prop}[thm]{ Proposition}
\newtheorem{lem}[thm]{Lemma}

\theoremstyle{definition}

\newtheorem{defin}[thm]{Definition}

\newtheorem*{rems*}{Remarks}
\theoremstyle{remark}
\newtheorem{rem}[thm]{Remark}

\newcommand{\eps}{\ensuremath{\varepsilon}}

\newtheorem{teor}{\hspace{12pt} Theorem}

\newtheorem{lema}[teor]{\hspace{12pt} Lemma}

\numberwithin{teor}{section}

\newcommand{\be}{\begin{enumerate}}
\newcommand{\ee}{\end{enumerate}}
\newcommand{\bi}{\begin{itemize}}
\newcommand{\ei}{\end{itemize}}
\newcommand{\bd}{\begin{description}}
\newcommand{\ed}{\end{description}}
\newcommand{\bec}{\begin{equation}}
\newcommand{\eec}{\end{equation}}
\newcommand{\ba}{\begin{array}}
\newcommand{\ea}{\end{array}}
\newcommand{\bt}{\begin{thm}}
\newcommand{\et}{\end{thm}}
\newcommand{\bdem}{\begin{proof}}
\newcommand{\edem}{\end{proof}}
\newcommand{\bl}{\begin{lema}}
\newcommand{\el}{\end{lema}}
\newcommand{\bnp}{\begin{rem}}
\newcommand{\enp}{\end{rem}}
\newcommand{\bde}{\begin{defin}}
\newcommand{\ede}{\end{defin}}
\newcommand{\bnod}{\begin{rem}}
\newcommand{\enod}{\end{rem}}
\newcommand{\bp}{\begin{prop}}
\newcommand{\ep}{\end{prop}}
\newcommand{\bco}{\begin{cor}}
\newcommand{\eco}{\end{cor}}

\newcommand{\nn}{\nonumber}
\newcommand{\lb}{\label}

\newcommand{\ccdot}{\, \cdot \,}

\newcommand{\gU}{\mathsf{U}}

\newcommand{\dc}{\dot{c}}

\newcommand{\dgamma}{\dot{\gamma}}

\DeclareMathOperator{\Ric}{Ric}

\DeclareMathOperator{\scal}{scal}

\DeclareMathOperator{\Or}{O}\DeclareMathOperator{\Iso}{Iso}
\DeclareMathOperator{\SO}{SO}

\DeclareMathOperator{\vol}{vol}

 \DeclareMathOperator{\diam}{diam}

\newcommand{\ml}{\langle}                     
\newcommand{\mr}{\rangle}                     

\newcommand{\tf}{\tilde{f}}

\newcommand{\bg}{\bar{g}}

\newcommand{\tb}{\tilde{b}}
\newcommand{\tD}{\widetilde{D}}
\newcommand{\ene}{\end{equation} }

\newcommand{\lG}{\mathsf{G}}

\newcommand{\tq}{\tilde{q}}

\def\l{\lambda}

\usepackage{bm}

\usepackage[draft]{optional}

\renewcommand{\(}{\left(}

\renewcommand{\>}{\right>}
\renewcommand{\)}{\right)}
\def\bal{\begin{align}}
\def\eal{\end{align}}

\numberwithin{equation}{section}
\def\be{\begin{equation}}
\def\ee{\end{equation}}

\def\vle{{\rm vol}}
\def\parcial#1#2{\frac{\partial #1}{\partial#2}}

\def\flecha{\longrightarrow}
\def\fle{\rightarrow}

\def\ds{\displaystyle}

\begin{document}

\markright{}


\pagestyle{myheadings}

\title{How to produce a Ricci Flow  via Cheeger-Gromoll exhaustion}

\author{ Esther Cabezas-Rivas and Burkhard Wilking}
\dedicatory{Dedicated to Wolfgang T. Meyer on the occasion of his 75th birthday}

\maketitle

\begin{abstract}
We prove short time existence for the Ricci flow on open mani\-folds 
of nonnegative complex sectional curvature. 
We do not require upper curvature bounds. 
By considering the doubling of  convex sets
contained in a  Cheeger-Gromoll convex exhaustion 
and solving  the singular initial value problem 
for the Ricci flow on these closed manifolds, we obtain a sequence 
of closed solutions of the Ricci flow with nonnegative 
complex sectional curvature  
which subconverge  to a solution of the Ricci 
flow 
on the open manifold.  Furthermore, we find an optimal volume growth condition which guarantees long time existence, and we give an analysis of the long time 
behaviour of the Ricci flow.
Finally, we construct an explicit example of an immortal nonnegatively curved solution of the Ricci flow with unbounded curvature for all time.
\end{abstract}

\section{Introduction and main results}

The Ricci flow was introduced by R.~Hamilton in \cite{ham3D} as a method to deform or evolve a Riemannian metric $g$ given on a fixed n-dimensional manifold $M$ according to the following partial differential equation:
\bec \lb{RF}
\parcial{}{t} g(t) = - 2 {\rm Ric}(g(t))
\eec
over a time interval $I \subset \mathbb R$, with the initial
condition $g(0) = g$. The first basic question, without which a theory about the Ricci flow does not even make sense, is to ensure that equation \eqref{RF} admits a solution at least for a short time. This was already completely settled for closed manifolds (i.e.~compact and without boundary) by Hamilton in \cite{ham3D}. In dimension 2, short time existence for a non-compact surface (which may be incomplete and with curvature unbounded above and below) was established by Giesen and Topping in \cite{GieTo1} using ideas from \cite{Topping}.

The non-compact case for $n \geq 3$, even asking the manifold to be complete, is much harder and in full generality appears to be hopeless: for instance, it is difficult to imagine how to construct a solution to \eqref{RF} starting at a manifold built by attaching in a smooth way spherical cylinders with radius becoming smaller and smaller (say of radius $1/k$ with $k \in \mathbb N$).  Hence to achieve short time existence one needs to prevent similar situations by adding extra conditions on the curvature. In this spirit, W.~X.~Shi proved in \cite{Shi1} that the Ricci flow starting on an open (i.e.~complete and non-compact) manifold  with bounded curvature  (i.e.~with $\sup_M |{\Rm}_g| \leq k_0<\infty$) admits a solution for a time interval $[0, T(n, k_0)]$ also with bounded curvature. 

Later on M.~Simon (cf.~\cite{Simon}), assuming further that the manifold has nonnega\-tive curvature operator ($\Rm_g \geq 0$) and is non-collapsing 
($\inf_M \vle_g \(B_g(\cdot, 1)\) \geq v_0 > 0$), was able to extend Shi's solution for a time interval $[0, T(n, v_0)]$, with curvature bounded above by $\frac{c(n, v_0)}{t}$ for positive times.  Although  $T(n, v_0)$ does not depend on an upper curvature 
bound, such a bound  is still an assumption needed to guarantee short time existence.

The present paper manages to remove any restriction on upper curvature bounds for  open manifolds
with nonnegative complex sectional curvature (see Definition~\ref{complex curv}) which, by Cheeger, Gromoll and Meyer \cite{ChGr, GroMe}, admit  an exhaustion by convex sets $C_\ell$. We are  able to construct 
a Ricci flow with nonnegative complex sectional curvature
on the closed manifold obtained by gluing two copies of  $C_\ell$
along the common boundary, and 
whose \lq initial metric' is the natural singular metric on the double.
By passing to a limit we obtain

\begin{mainthm} \lb{mainT}
Let $(M^n, g)$ be an open manifold with nonnegative (and possibly unbounded) complex sectional curvature ($K^\ce_g \geq 0$).
Then there exists a constant $\mc T$ depending on $n$ 
and $g$ such that \eqref{RF} has a smooth solution 
on the interval $[0, \mc{T}]$, with $g(0) = g$ and with
 $g(t)$ having nonnegative complex sectional curvature.
\end{mainthm}

Using that by Brendle \cite{BreHar} the trace 
Harnack inequality in \cite{HamHar} holds for compact mani\-folds with $K^\ce \geq 0$, it follows that the above solution 
on the open manifold satisfies the trace Harnack estimate as well.
This solves an open question posed by Chow, Lu and Ni \cite[Problem 10.45]{CLN}.

The proof of Theorem~\ref{mainT} is  easier if $K^\ce_g > 0$ since then by Gromoll and Meyer \cite{GroMe} $M$ is diffeomorphic 
to $\R^n$. In the general case, we need  additional tools; for instance, we prove the following result, which  extends a theorem by Noronha \cite{N} for 
manifolds with $\Rm_g \geq 0$.
\begin{mainthm} \lb{splitting}
Let $(M^n, g)$ be an open, simply connected Riemannian manifold with 
nonnegative complex sectional curvature. Then $M$ splits isometrically as  $\Sigma \times F$, where $\Sigma$ is the $k$-dimensional soul of $M$ and $F$ is diffeomorphic to $\re^{n - k}$. 
\end{mainthm}
\noindent In the nonsimply connected case $M$ is
diffeomorphic to a flat Euclidean vector bundle over the soul.
Thus combining with the classification in \cite{BrSch2} of compact manifolds 
with $K^\ce \geq 0$,
 we deduce that any open manifold of $K^\ce \geq 0$ admits a complete nonnegatively curved 
locally symmetric metric $\hat g$, i.e. $K_{\hat g}\ge 0$, $\nabla R_{\hat g}\equiv 0$.

It is not hard to see that, given any open manifold $(M,g)$ with
 bounded curvature and $K_g^\ce > 0$, for any closed
discrete countable subset $S\subset M$ one can find
a deformation $\bg $ of $g$
 in an arbitrary small neighborhood 
$U$ of $S$ such that $\bg$ and $g$ are $C^1$-close, 
 $(M,\bar g)$ has  unbounded curvature and $K^\ce_{\bg} > 0$.
The following result, which is very much in spirit of Simon
\cite{Simon}, shows that this sort of local deformations 
will be smoothed out instantaneously by our  Ricci flow.
\begin{maincor}\label{cor: noncollapsed} Let $(M^n,g)$ be an open manifold with $K_g^\ce \geq 0$.
If 
\bec \lb{non_col_hyp}
 \inf\bigl\{\vle_g(B_g(p,1)): p\in M\bigr\}=v_0>0,
\eec
then the curvature of $(M,g(t))$ is bounded above by $\tfrac{c(n,v_0)}{t}$ 
for $t\in (0,\mc T(n,v_0)]$.
\end{maincor}

\noindent In the case of a  nonnegatively curved surface, this volume condition is always satisfied (see \cite{CK}), so any such surface can be deformed by \eqref{RF} to one with bounded curvature. 
For $n\ge 3$ the lower volume bound 
in Corollary~\ref{cor: noncollapsed} is essential:

\begin{mainthm}\label{thm: exa}
a) There is an immortal $3$-dimensional nonnegatively 
curved complete  Ricci flow 
$(M,g(t))_{t\in[0,\infty)}$ with unbounded curvature 
for each $t$. 

b) There is  an immortal  $4$-dimensional  complete Ricci flow 
$(M,g(t))_{t\in[0,\infty)}$ with positive curvature operator 
such that the curvature of $(M,g(t))$ 
is bounded if and only if $t\in [0,1)$.                           
\end{mainthm}
\noindent Higher dimensional examples can be obtained by crossing 
with a Euclidean factor. 
Part b) shows that even if the initial metric has bounded curvature 
one can run into metrics with unbounded curvature.
The following result gives a precise lower bound on the existence time for \eqref{RF} in terms of supremum of the volume of balls, instead of infimum as in Corollary \ref{cor: noncollapsed} and \cite{Simon}. 
We emphasize that this is new even in the case 
of initial metrics of bounded curvature.
\begin{maincor}\label{thm: immortal}
In each dimension there is a universal  constant $\eps(n) > 0$ 
such that for each complete manifold
 $(M^n,g)$ with $K^\ce_g \geq 0$
the following holds: If  we put 
\[
\mc T:=\eps(n)\cdot \sup\Bigl\{\tfrac{\vle_g(B_g(p, r))}{r^{n-2}}\bigm|  p\in M, r>0\Bigr\}\in (0,\infty],
\]
then any complete maximal solution of   
Ricci flow $(M,g(t))_{t\in [0,T)}$  with $K^\ce_{g(t)} \geq 0$ and $g(0)=g$
 satisfies $\mc T\le T$. 
\end{maincor}

\noindent If $M$ has a volume growth larger than $r^{n-2}$, 
then Corollary \ref{thm: immortal} ensures the existence of an immortal solution. Previously (cf. \cite{SS}) long time existence was only known in the case of Euclidean volume growth under the stronger assumptions $\Rm_g \geq 0$ and bounded curvature. We highlight that our volume growth condition cannot be further improved: indeed, as the Ricci flow on the metric product $\Sph^2\times \R^{n-2}$
exists only for a finite time, 
the power $n-2$  is optimal. For $n = 3$ we can even determine 
exactly the extinction time depending on the structure of the manifold:
\begin{maincor}\label{thm: immortal3}
Let $(M^3, g)$ be an open manifold with $K_g \geq 0$ and soul $\Sigma$. If  $(M,g(t))_{t\in [0,T)}$  is a maximal complete solution of \eqref{RF} 
 with $g(0) = g$ and $K^\C_{g(t)}\ge 0$, then
\[
T= \left\{\ba{ll} \tfrac{\mathrm{area}(\Sigma)}{4\pi\chi(\Sigma)} & \text{ if } \quad \dim \Sigma = 2 \\
\infty & \text{ if } \quad \dim \Sigma = 1 \\
 \frac1{8 \pi} \lim_{r\to \infty} \frac{\vle_{g}\(B_{g}(p, r)\)}{r} & \text{ if } \quad  \Sigma = \{p_0\} \ea \right..
\]
In the case $\Sigma = \{p_0\}$, if  $T < \infty$, then $(M,g)$ is asymptotically cylindrical and $(M,g(t))$ has  bounded curvature for $t>0$.
\end{maincor}

By Corollary~\ref{thm: immortal} 
a finite time singularity $T$ on open manifolds with $K^\C\ge 0$ 
 can only occur if the manifold collapses uniformly as 
$t\to T$. For immortal solutions we 
will also give an analysis of the long time behaviour of the flow: 
In the case of an initial metric with Euclidean volume 
growth we remark that a result of Simon and Schulze \cite{SS} can be adjusted 
to see that a suitable rescaled Ricci flow 
subconverges to an expanding soliton, see Remark~\ref{rem: euclidean}. 
If the initial manifold does not have Euclidean volume growth, 
then by Theorem~\ref{thm: steady} 
any immortal solution can be rescaled suitably 
so that it subconverges to a steady soliton (different from the Euclidean space).

\section{Structure of the paper and strategy of proof}

Section \ref{sec:back} contains the background material that we use repeatedly throughout the paper. The definition of nonnegative complex 
sectional curvature, which implies nonnegative sectional curvature and has the advantage to be invariant under the Ricci flow, can be found in subsection \ref{sec:PCSC}. 
Subsection \ref{subsec: chgr} is about the basics of open nonnegatively curved manifolds.

Section \ref{sec:pos-case} carries out the proof of Theorem \ref{mainT} for the particular case of a mani\-fold $(M, g)$ with $K_g^\ce > 0$, which is an easier scenario since 
there is a {\em smooth} strictly convex proper 
function $\beta: M \fle [0, \infty[$. The idea (developed within the proof of Proposition \ref{IniAp}) is 
to show that the doubling $D(C_i)$ of the compact 
 sublevel  $C_i=\beta^{-1}([0,i])$ 
admits a metric with $K^\ce \geq 0$.
We actually prove that after replacing
$C_i$ by the graph of a convex function
defined on $C_i$ (a reparametrization of $\beta$) the doubling 
is a smooth closed manifold $(M_i, g_i)$ with $K^\ce_{g_i} > 0$. 
The sequence $(M_i,g_i)$ converges to $(M,g)$.  
The key is now to establish two important properties
for the Ricci flows of $(M_i,g_i)$: (1) there is a lower bound (independent of $i$) for the maximal times of existence $T_i$ (Proposition \ref{LBT}), and (2) we can find arbitrarily large balls around the soul point $p_0$ where the curvature has an upper bound of the form $C/t$ (here $C$ depends on the distance to $p_0$, see Proposition \ref{Int_BC}).
The crucial tool for (1) is a result by Petrunin (Theorem \ref{Pet}) which also allows to conclude that the evolved unit balls around the soul are uniformly non-collapsed (Corollary \ref{Cor_vol}). For the proof of (2) we use a fruitful point-picking technique by Perelman \cite{P1},  and we also need to obtain an improved version of  11.4 in \cite{P1} (Lemma
 \ref{AVR0}). All these results ensure that we can perform suitable compactness arguments to prove Theorem \ref{mainT} for the positively curved case (Theorem \ref{mainT_pos}).

Several additional difficulties arise when we 
just assume $K^\ce_g \geq 0$. For instance, the soul is not necessarily a point.
A harder issue is  
that the sublevels of a Busemann function
$C_\ell=b^{-1}((-\infty,\ell])$ have non-smooth boundary.
Thus there is no obvious smoothing of the doubling 
$D(C_\ell)$ with $K^\ce \geq 0$. Section \ref{sec:misc} gathers the technical results we will need to apply in Section \ref{sec:gral} to overcome the extra complications of the general case of Theorem \ref{mainT}: we prove Theorem \ref{splitting}, which essentially reduces the problem to the 
situation where the soul is a point; we establish  two estimates for abstract solutions of a Riccati equation 
which are used later to give a quantitative estimate of the convexity of
 the sublevels $C_{\ell}$  in terms of the curvature (Lemmas \ref{Ric_comp_cons} and \ref{Ric_comp}); in Proposition \ref{curv_est_NN} we get curvature estimates in terms of volume and lower sectional curvature; finally, we include a technical result (Lemma \ref{lem: smoothing}) about how to perform a smoothing process for $C^{1,1}$ hypersurfaces with bounds on the 
principal curvatures in the support sense by $C^\infty$ hypersurfaces where the bounds change  with an arbitrarily small error.

All the auxiliary results from Section \ref{sec:misc}
 are employed in Section \ref{sec:gral} to give a complete proof of Theorem \ref{mainT}. First, we prove upper and lower estimates for the Hessian of $d^2(\cdot, C_\ell)$ (see Proposition \ref{AR32} and Corollary \ref{est_Hesf}), and then we reparametrize such a  distance function to give a sequence of functions  whose graphs $D_{\ell, k}$, after a smoothing process, give $C^\infty$ closed manifolds converging to the double $D(C_\ell)$. The sets $D_{\ell, k}$ are not anymore convex, but we have a precise control on the complex sectional curvatures of the induced metrics $g_{\ell, k}$ (see Proposition \ref{smoothing}). In Proposition \ref{prop:curv_est} we prove that such curvature control survives for some time for the Ricci flows starting on $(D_{\ell, k}, g_{\ell, k})$. As a consequence 
we get, for all large $\ell$,  a solution of the Ricci 
flow on $D(C_\ell)$ with $K^\ce \geq 0$, and
whose \lq initial metric' is the natural singular metric
on the double. The rest of the proof is then 
essentially analogous to Section~\ref{sec:pos-case}.

Corollary~\ref{cor: noncollapsed}, \ref{thm: immortal} and 
\ref{thm: immortal3} are proved in Section~\ref{sec: applications}
 and Theorem \ref{thm: exa} is proved in Section~\ref{sec: immortal}

We end with three appendices containing additional background about open nonnegatively curved manifolds (Appendix \ref{AppA}), results for convex sets in Riemannian manifolds (Appendix \ref{AppB}) and results about smooth convergence and curvature estimates for the Ricci flow (Appendix \ref{AppC}).

\section{Basic background material} \lb{sec:back}

\subsection{About the relevant curvature condition} \lb{sec:PCSC}

We first need to introduce
\bde \lb{complex curv}
Let $(M^n, g)$ be a Riemannian manifold, and consider its complexified tangent bundle $T^\ce M:= TM\otimes \mathbb{C}$.
We extend the curvature tensor $\Rm$ and the metric $g$ at $p$
to $\C$-multilinear maps $\Rm\colon (T^\ce_p M)^4\rightarrow \C$,
$g\colon (T^\ce_p M)^2\rightarrow \C$. 
 The complex sectional curvature of a 2-dimensional 
complex subspace $\sigma$ of $T_p^\ce M$ is defined by
$$K^{\mathbb{C}}(\sigma) =  \Rm(u, v, \bar v, \bar u) = g(\Rm(u \wedge v), \overline{u \wedge v}),$$
where $u$ and $v$ form any unitary basis for $\sigma$, i.e.
$g(u,\bar u)=g(v,\bar v)=1$ and $g(u,\bar v)=0$.
We say $M$ has nonnegative complex sectional curvature if $K^\C\ge 0$. 

The manifold has nonnegative isotropic 
curvature if $K^\C(\sigma)\ge 0$ for any isotropic plane $\sigma\subset T_p^\C M$, 
i.e. $g(v,v)=0$ for all $v\in \sigma$.

\ede

\bnod \lb{RemPCSC}
Here we collect some relevant features known about the above curvature condition (see \cite{BrSch1} and \cite{NiWol} for the proofs).
\begin{enumerate}
\item[(a)] If $g$ has strictly (pointwise) $1/4$-pinched sectional curvature, then $K^{\mathbb C}_g > 0$. 
\item[(b)] Nonnegative curvature operator ($R_g \geq 0$) implies $K^\ce_g \geq 0$, which in turn gives nonnegative sectional curvature ($K_g \geq 0$). For $n\le 3$ the converse holds.
\item[(c)] $K_{(M,g)}^{\mathbb C} \geq 0$ if and only if $(M, g) \times \re^2$ has nonnegative isotropic curvature. 
\item[(d)] The positivity and nonnegativity of $K^{\mathbb C}$ is preserved under the Ricci flow.
\item[(e)] Let $(M, g)$ be closed with $K^\ce_g > 0$. Then $g$ is deformed by the normalized Ricci flow to a metric of positive constant sectional curvature, as time goes to infinity.
\end{enumerate}
\enod

\begin{prop}\lb{hol}
Let $(M^n, g)$ be closed with $K_g^\ce \geq 0$. If $M$ 
is homeomorphic to a sphere, then the Ricci flow $g(t)$ with $g(0) = g$ has $K_{g(t)}^\ce > 0$ for any $t > 0$.
\end{prop}

\bdem  Clearly $g$ cannot be Ricci flat as this 
would give a flat metric on a sphere.
Moreover, since $M$ is a sphere the metric is 
irreducible and neither K\"ahler nor Quaternion-K\"ahler. 
If $(M,g)$ is a locally symmetric space 
we could use a result of \cite{Bor} 
to see that $(M,g)$ is round. 
Combining all this with the holonomy classification 
of Berger \cite{Berh} we deduce that $g$ as well as $g(t)$ has $\SO(n)$ 
holonomy. Now the statement follows from the proof of \cite[Proposition 10]{BrSch2}.
\edem

\subsection{Cheeger-Gromoll convex exhaustion}\label{subsec: chgr}

Let $(M,g)$ be a nonnegatively curved open manifold. 
A ray is a unit speed geodesic $\gamma\colon [0,\infty)\rightarrow M$ 
such that $\gamma_{[0,s]}$ is a minimal geodesic for all $s>0$.
Fix $o \in M$, and consider the set of rays 
$$\mc R = \{\gamma: [0, \infty) \fle M : \gamma\text{ is a ray with }  \gamma(0) = o\}.$$ 
 Recall that
$$b =\sup_{\gamma \in \mc R}\Big\{ \lim_{s \to \infty} \big(s - d_g(\gamma(s), \cdot)\big)\!\Big\}$$
 is called the Busemann function of $M$. 
By the work of  Cheeger, Gromoll  and Meyer \cite{GroMe, ChGr} $b$
 is a convex function, that is, for any geodesic $c(s)\in M$ the function 
$s\mapsto b\circ c(s)$ is convex. 
Equivalently one can say  that $b$ 
satisfies $\nabla^2b\ge 0$ in the support sense (cf.~Definition \ref{def_spt}).

The following properties of the sublevels
$C_\ell:=b^{-1}((-\infty,\ell])$ 
will be used throughout the paper:
\begin{enumerate}
\item Each $C_\ell$ 
 is a totally convex compact set,
\item ${\rm dim} \, C_\ell = n$ for all $\ell > 0$, $\cup_{\ell > 0} C_\ell = M$,
\item $s < \ell$ implies $C_s \subset C_\ell$ and $C_s = \{x \in C_\ell\, :\, d_g(x, \partial C_\ell) \geq \ell - s\}$,
\item  each $C_\ell$, $\ell >0$, has the structure of an embedded submanifold of $M$ with smooth totally geodesic interior and (possibly non-smooth) boundary.
\end{enumerate}
The family ${C_\ell}$ is part of the
Cheeger-Gromoll convex exhaustion used for the soul construction (see some more details in Appendix \ref{AppA}). 
For us only the structure of $C_\ell$ for $\ell\to \infty $ 
is of importance.
If $(M,g)$ has positive rather than nonnegative sectional curvature,
then $\nabla^2 e^b> 0$ holds in the support sense.
By a local smoothing procedure  one can then show

\bt[Greene-Wu, \cite{GW1}] \lb{GW1t}
If $(M^n, g)$ is an open manifold with $K_g > 0$, then there exists a smooth  proper strictly convex function $\beta\colon M\rightarrow [0,\infty[$.
\et
\noindent The main reason why the proof of Theorem~\ref{mainT} is quite 
a bit easier in the positively curved case is this theorem. 
In the nonnegatively curved case we will have to work 
with the sublevels of the Busemann function instead.

\section{Manifolds with positive complex sectional curvature}\lb{sec:pos-case}
\subsection{Approximating sequence for the initial condition}
Let $(M^n, g)$ be an open manifold with $K^\ce_g > 0$. On $M$ we can consider a function $\beta$ as described in Theorem \ref{GW1t}. 
Since $\beta$ is proper, the global minimum is attained
and we may assume that its value is $0$. 
Since $\beta$ is strictly convex, $\beta^{-1}(0)$ 
consists of a single point $p_0$, and clearly $p_0$ 
is the only critical point of $\beta$. Hence the sublevel set
\bec \lb{defCi}
C_i = \{x \in M\, :\, \beta(x) \leq i\}
\eec
is a convex set with a smooth boundary for all $i>0$. 
Recall that $\beta$ is obtained essentially from a smoothing of a Busemann function 
$b$. Thus we may assume that for each $i$ there is some $\ell_i$ 
so that $C_i$ has Hausdorff distance $\le 1$ to $b^{-1}((-\infty,\ell_i])$. 

The goal is to construct a pointed sequence of closed manifolds converging to $(M, g, p_0)$. The first attempt would be to consider the double $D(C_i)$ of $C_i$ (which is obtained by gluing together two copies of $C_i$ along the identity map of the boundary). However,  $D(C_i)$ is usually not a
 smooth Riemannian manifold. To overcome this, we adapt to our setting ideas from \cite{Kron, Gui} which roughly consist in modifying the metric in a small inner neighborhood of the boundary $\partial C_i$ to form a cylindrical end so that the gluing is well defined.

\begin{prop}\lb{IniAp}
Let $(M^n, g)$ be an open manifold  with $K^\ce_g > 0$ and soul point $p_0$. Then there exists a collection $\{(M_i, g_i, p_0)\}_{i \geq 1}$ of smooth closed $n$-dimensional pointed manifolds with $K_{g_i}^\ce > 0$ satisfying
$$(M_i, g_i, p_0) \longrightarrow (M, g, p_0) \qquad \text{ as} \quad i \to \infty$$
in the sense of the smooth Cheeger-Gromov convergence (cf.~Definition \ref{def_conv}).
\end{prop}

\bdem
For each fixed $i$, consider $C_i$ as in \eqref{defCi}. The goal is to modify the metric $g|_{C_i}$ within $C_i \setminus C_{i - \eps}$. For that aim, let us choose any real function $\varphi_i$ such that
\bi
\item[(a)] $\varphi_i$ is smooth on $(-\infty, i)$ and continuous at $i$,
\item[(b)] $\varphi_i\equiv 0$ on $(-\infty, i - \eps]$ and $\varphi_i(i) = 1$. 
\item[(c)] $\varphi'_i$, $\varphi''_i$ are positive on $(i - \eps, i)$,
\item[(d)] $\varphi^{-1}_i$ has all left derivatives vanishing at $1$,
\ei
By (d) the derivative $\varphi_i'(s)$ tends to $\infty$ 
for $s\to i$. 
Now take $u_i := \varphi_i \circ \beta$ and 
put 
\begin{eqnarray*}
 G_i &=& \{(x, u_i(x)) \, : \, x \in C_i\}\\
 \tilde G_i&=& \{(x, 2-u_i(x)) \, : \, x \in C_i\}
\end{eqnarray*}
Note that the submanifolds $G_i$ and $\tilde{G}_i$ 
are isometric and (d) ensures that they paste smoothly together to a $C^\infty$ 
closed hypersurface 
$D(C_i)=G_i\cup \tilde{G}_i$ of $M\times \R$.

Clearly the induced metric 
of $G_i$ can be regarded as a deformation 
of the metric on $C_i$. 
Given the properties of $\varphi_i$ and $\beta$, 
it is straightforward to check 
that $u_i$ is a convex function. 
Using this and that $M\times \R$ has nonnegative complex sectional curvature, we deduce that $(M_i,g_i):=D(G_i)$ has nonnegative complex sectional curvature 
as well.

Notice that $C_{i - \eps}$ can be seen as a subset of $M_i$ for all $i >0$, which immediately implies that $(M_i,g_i,p_0)$ converges to $(M,g,p_0)$ 
in the Cheeger-Gromov sense. 
We now use the short time existence of the Ricci flow 
on $M_i$ (cf.~\cite{ham3D}), and choose $t_i>0$ so small 
that $(M_i,g_i(t_i),p_0)$ still converges to $(M,g,p_0)$.

Since $M_i$ is a topological sphere, we can employ 
Proposition~\ref{hol} to conclude that 
$K^\ce_{g_i(t_i)} > 0$. 
Thus $g_{i,new}=g_i(t_i)$ is a solution of our problem.
\edem

\subsection{Ricciflowing the approximating sequence}

Consider $\{(M_i, g_i, p_0)\}$ the sequence of closed, positively curved manifolds obtained above. For each fixed $i$ we can construct a Ricci flow $(M_i, g_i(t))$ defined on a maximal time interval $[0, T_i)$, with $T_i < \infty$, and such that $g_i(0) = g_i$. 

\subsubsection{A uniform lower bound for the lifespans}

 The first difficulty to address is that the curvature of $g_i$ will tend to infinity as $i \to \infty$, so it may happen that the maximal time of existence of the flow $T_i$ goes to zero as
$i$ tends to infinity. Then our next concern is to prove that the times $T_i$ admit a uniform lower bound $T_i \geq \mc T > 0$ for all $i$. The key to achieve such a goal is to estimate the volume growth of unit balls around $p_0$.
For such an estimate, we make a strong use of
\bt[Petrunin, cf.~\cite{Pet}] \lb{Pet} Let $(M^n, g)$ be a complete manifold with $K_g \geq -1$. Then for any $p$ in $M$
$$\int_{B_g(p, 1)} \scal_g \, d\mu_g \leq C_n,$$
for some constant $C_n$ depending only on the dimension.
\et

\begin{prop}\lb{LBT}
Let $(M, g)$ and $(M_i, g_i, p_0)$ be as in Proposition \ref{IniAp}. Then there exists a constant ${\mc T} > 0$, depending on $n$,  and $V_0:= \vle_{g} \(B_{g}(p_0, 1)\)$ (but independent of $i$), such that the Ricci flows $(M_i, g_i(t))$ with $g_i(0) = g_i$ are defined on $[0, {\mc T}]$, and satisfy $K^\ce_{g_i(t)} > 0$ for all $t\in [0, \mc T]$.
\end{prop}

\bdem
For each $i$, $(M_i, g_i)$ is a closed $n$-manifold; so the classical short time existence theorem in \cite{ham3D} ensures that there exists some $T_i > 0$ and a unique maximal Ricci flow $(M_i, g_i(t))$ defined on $[0, T_i)$ with $g_i(0) = g_i$. Moreover, $K^\ce_{g_i(t)} > 0$, since this is true for $t = 0$ by Proposition \ref{IniAp}, and positive complex sectional curvature is preserved under the Ricci flow (cf.~Remark \ref{RemPCSC} (d)).

Observe that ${\rm Ric}_{g_i(t)} > 0$ implies $B_{g_i(0)}(p_0,1)\subset B_{g_i(t)}(p_0,1)$.
Using the evolution equation of the Riemannian volume element $d\mu_{g_i(t)}$ under the Ricci flow and applying Theorem \ref{Pet}, we get
\bec \lb{Pet_I}
\parcial{}{t} \vle_{g_i(t)} \(B_{g_i(0)} (p_0, 1)\) = - \int_{B_{g_i(0)} (p_0, 1)} \scal_{g_i(t)} \, d\mu_{g_i(t)} \geq -C_n.
\eec
Hence
\bec \lb{vle_ftc}
\vle_{g_i(t)} \(B_{g_i(0)} (p_0, 1)\) - \vle_{g_i(0)} \(B_{g_i(0)} (p_0, 1)\) \geq -C_n t.
\eec

On the other hand, as $K^\ce_{g_i} > 0$, we know (cf.~Remark \ref{RemPCSC} (e)) that the volume of  $(M_i, g_i(t))$ vanishes completely at the maximal time $T_i$ so that
$$T_i \geq \frac{\vle_{g_i(0)} \(B_{g_i(0)} (p_0, 1)\)}{C_n}\stackrel{i\to\infty}{\longrightarrow} \frac{\vle_{g} \(B_{g} (p_0, 1)\)}{C_n}=: 2 \mc T.$$
\edem

As a consequence, we obtain a uniform (independent of $t$ and $i$) lower bound for the volume of unit balls centered at the soul point:

\bco \lb{Cor_vol}
For the sequence of pointed Ricci flows $(M_i, g_i(t), p_0)_{t \in [0, \mc{T}]}$  from Proposition \ref{LBT}, we can find a constant $v_0 = v_0(n, V_0)$ satisfying
$$\vle_{g_i(t)} \(B_{g_i(t)} (p_0, 1)\) \geq v_0 >0 \qquad \text{for any } \quad t \in [0,\mc{T}].$$
\eco

\bdem
Using again \eqref{Pet_I} and $t \leq \mc T := V_0/(2 C_n)$, we obtain
\begin{align*}
\vle_{g_i(t)} \(B_{g_i(t)} (p_0, 1)\) &\geq \vle_{g_i(t)} \(B_{g_i(0)} (p_0, 1)\) \geq \vle_{g_i(0)} \(B_{g_i(0)} (p_0, 1)\) - C_n t 
\\ & \geq \frac3{4} V_0 - C_n \mc T = \frac{V_0}{4} =: v_0 > 0.
\end{align*}
\edem

\subsubsection{Interior curvature estimates around the soul point}

The first step in order to get a limiting Ricci flow starting on $(M, g)$ from the sequence $(M_i, g_i(t))$ is to obtain uniform (independent of $i$, but maybe depending on time and distance to $p_0$)
 curvature estimates. We first need an improved version of \cite[11.4]{P1}:
\begin{lem}\lb{AVR0}
Let $(M^n, g(t))$, $t\in (-\infty, 0]$ be an open, non-flat ancient solution of the Ricci flow. Assume further that $g(t)$ has bounded curvature operator, and that $K^\ce_{g(t)} \geq 0$.
Then  $\ds\lim_{r \to \infty} \tfrac{\vle_{g(t)}\( B_{g(t)}(\ccdot, r)\)}{r^n}$ vanishes for all $t$.
\end{lem}
\bdem
The $k$-noncollapsed assumption from 11.4 in \cite{P1} was already removed in \cite{Ni}. So it only remains to ensure that we can relax $\Rm_{g(t)} \geq 0$ to $K^\ce_{g(t)} \geq 0$. One can go through the original proof and check that the only instances in which one needs the full $\Rm_{g(t)} \geq 0$ (instead of just $K_{g(t)} \geq 0$) is when one applies Hamilton's trace Harnack inequality (cf.~\cite{HamHar}) or Hamilton's strong maximum principle in \cite{ham4D}. But under our weaker assumption we can replace them by Brendle's trace Harnack in \cite{BreHar} and the strong maximum principle of Brendle and Schoen \cite[Proposition 9]{BrSch2} (see also the Appendix of \cite{Wi}). The rest of the proof proceeds verbatim as the original one.
\edem

\begin{prop}\lb{Int_BC}
Consider the Ricci flows $(M_i, g_i(t))$, with $t \in [0, \mc T]$, coming from Proposition \ref{LBT}. 
For any $D > 0$ there exists a constant $C_D > 0$ such that
$$\scal_{g_i(t)} (x) \leq \frac{C_D}{t} \qquad \text{for all }  \quad i \geq 1, \quad x \in B_{g_i(t)}(p_0, D) \quad \text{and} \quad t \in (0, \mc T].$$
\end{prop}

\bdem
Assume, on the contrary, that we can find a constant $D_0 > 0$ so that there exist indices $i_k \geq 1$ (for brevity, let us denote as $(M_k, g_k(t))$ the corresponding subsequence $(M_{i_k}, g_{i_k}(t))$), and sequences of  times $t_k \in (0, \mc T)$ and points $p_k \in B_k(p_0, D_0)$ (hereafter $B_k = B_{g_k(t_k)}$, $\scal_k = \scal_{g_k(t_k)}$ and $d_k = d_{g_k(t_k)}$) satisfying
\bec \lb{cont_hyp}
\scal_k (p_k) > 4^k/t_k.
\eec

{\it Claim 1.} We can find a sequence of points $\{\bar p_k\}_{k \geq k_0}$ which satisfy \eqref{cont_hyp} and 
\begin{align*}
\scal_{g_k(t)}(p) \leq 8 \, \scal_k(\bar p_k)  \quad \text{ for all} \quad \left\{\ba{l}  p\in B_k\bigl(\bar p_k, \tfrac{k}{\sqrt{\scal_k(\bar p_k)}}\bigr), \medskip \\  t\in \big[t_k - \tfrac{k}{\scal_k(\bar p_k)}, t_k\big]  \ea\right.
\end{align*}
with $d_k(\bar p_k, p_0) \leq D_0 + 1$.

Notice that it is enough to prove
\begin{align}
\scal_k (p) \leq 4 \, \scal_k(\bar p_k)  \quad \text{ for all } \quad p \in B_k\big(\bar p_k, \tfrac{k}{\sqrt{\scal_k(\bar p_k)}}\big), \lb{Cl1_1}
\end{align}
with $d_k(\bar p_k, p_0) \leq D_0 + 1$.
In fact, as $K^\ce_{g_k(t)} \geq 0$,  we can apply the  trace Harnack inequality in \cite{BreHar} (which, in particular, gives $\parcial{}{t} (t \scal_{g(t)}) \geq 0$). This yields for any $t \in \big[t_k - \tfrac{k}{\scal_k(\bar p_k)}, t_k\big]$  
$$\scal_{g_k(t)} \leq \frac{t_k}{t} \scal_k \leq \frac{t_k}{t_k - k/\scal_k(\bar p_k)} \scal_k < 2 \scal_k,$$
where we have used that \eqref{cont_hyp} implies $\tfrac{k}{\scal_k(\bar p_k)} <  t_k \, \tfrac{k}{4^k}  < \tfrac{t_k}{4}.$

So our goal is to find $\bar p_k$ satisfying \eqref{cont_hyp} and \eqref{Cl1_1}. If \eqref{Cl1_1} does not hold for $\bar p_k = p_k$, it means that there exists a point $x_1 \in B_k\big(p_k, \tfrac{k}{\sqrt{\scal_k(p_k)}}\big)$ such that $\scal_k(x_1) > 4 \scal_k(p_k)$. Next, check if  \eqref{Cl1_1} holds for $\bar p_k = x_1$, namely, if  
$$\scal_k(p) \leq 4 \, \scal_k(x_1) \qquad \text{ for all } \quad p\in B_k\big(x_1, \tfrac{k}{\sqrt{\scal_k(x_1)}}\big).$$
In case this is not satisfied, we iterate the process and, accordingly, we construct a sequence of points $\{x_j\}_{j \geq 2}$ such that
$$x_j \in B_k\Bigl(x_{j-1}, \tfrac{k}{\sqrt{\scal_k(x_{j-1})}}\Bigr)
\quad \text{and} \quad \scal_k(x_j) > 4  \, \scal_k(x_{j -1}).$$

Thus $(x_j)_{j\in \N} $ is a Cauchy sequence and
a straightforward computation shows that 
it stays in the relatively compact ball $B_k(p_0, D_0 + 1)$. Because of
$$\lim_{j \to \infty} \scal_k(x_j) = \infty$$
this gives a contradiction. In conclusion, there exists $\ell \in \mathbb N$ such that $\bar p_k$ can be taken to be $x_\ell$.

Now from Claim 1 it follows that $B_{k}(p_0, r) \subset B_k(\bar p_k, r + D_0 + 1)$. Then for $r\in [D_0 + 3/2, D_0 + 2]$, using Corollary \ref{Cor_vol} with $\vle_k = \vle_{g_k(t_k)}$, we get
\begin{align*}
\frac{\vle_k \(B_k(\bar p_k, r)\)}{r^n} &\geq \frac{\vle_k \(B_k(p_0, r - D_0 - 1)\)}{r^n} 
 \geq \vle_k \(B_k(p_0, 1)\) \(\frac{r - D_0 - 1}{r}\)^n
 \\& \geq   \frac{v_0/2^n}{(D_0 + 2)^n} =: \tilde v_0 > 0.
\end{align*}
Next, Bishop-Gromov's comparison theorem ensures that the above conclusion is true even for smaller radius:
\bec \lb{vol_br}
\frac{\vle_k \(B_k (\bar p_k, r)\)}{r^n} \geq \tilde v_0 > 0 \quad \text{ for} \quad 0 < r \leq D_0 + 2.
\eec

After a parabolic rescaling of the metric
$$\tilde g_k(s) = Q_k g(\ccdot, t_k + s Q_k^{-1}) \qquad \text{ for } \qquad Q_k = \scal_k(\bar p_k) > 4^k/\mc T,$$
using $K_{\tilde g_k(s)} > 0$, Claim 1 says that for $k \geq k_0$
\begin{align}
|\Rm|_{\tilde g_k(s)} \leq \scal_{\tilde g_k(s)} \leq 8  \quad \text{ on} \quad   B_{\tilde g_k(0)}(\bar p_k, k) \quad \text{ for all } \quad s\in [- k, 0].  
\lb{Cl1_resc}
\end{align}

In addition, as the volume ratio in \eqref{vol_br} is scale-invariant, we have
\bec \lb{Vol_rat}
 \frac{\vle_{\tilde g_k(0)} \(B_{\tilde g_k(0)} (\bar p_k, r)\)}{r^n} \geq \tilde v_0 > 0 \quad \text{ for} \quad 0 < r \leq (D_0 + 2) \sqrt{Q_k}.
\eec 
Combining this and \eqref{Cl1_resc} with Theorem \ref{ChGrT} gives
$${\rm inj}_{\tilde g_k(0)}(\bar p_k) \geq   c(n, \tilde v_0).$$
Joining the above estimate to \eqref{Cl1_resc}, we are in a position to apply  Hamilton's compactness (cf.~Theorem \ref{locHCT}) to the pointed sequence
$$(M_k, \tilde g_k(s), \bar p_k), \qquad s \in [-k, 0]$$
to obtain a subsequence converging, in the smooth Cheeger-Gromov sense to a smooth limit solution of the Ricci flow
$$(M_\infty, g_\infty(t), p_\infty) \qquad t\in (-\infty, 0]$$
which is complete, non compact (since the diameter with respect to $\tilde g_k(s)$ tends to infinity with $k$ because $Q_k \to \infty$), non-flat (as $\scal_{g_\infty(0)} (p_\infty) = 1$),
of bounded curvature (more precisely, $|\Rm|_{g_\infty(t)} \leq 8$ on $M_\infty \times (-\infty, 0]$), and with $K_{g_\infty(t)}^\ce \geq 0$. Moreover, from \eqref{Vol_rat} and volume comparison, we have
$$\tilde v_0 \leq \frac{\vle_{g_\infty(0)}\( B_{g_\infty(0)} (p_\infty, r)\)}{r^n} \leq \omega_n \qquad \text{for all} \quad r > 0.$$
Therefore, the limit of the volume ratio as $r \to \infty$ also lies between two positive constants, which contradicts Lemma \ref{AVR0}.
\edem

\subsection{Proof of short time existence for the positively curved case}

\bt \lb{mainT_pos}
Let $(M^n, g)$ be an open manifold with $K^\ce_g > 0$ (and possibly unbounded curvature). Then there exists $\mc T > 0$ and a sequence of closed Ricci flows $(M_i, g_i(t), p_0)_{t \in [0, \mc T]}$ with $K^\ce_{g_i(t)} > 0$ which converge in the smooth Cheeger-Gromov sense to a complete limit solution of the Ricci flow
$$(M, g_\infty(t), p_0) \qquad \text{for } \quad t\in [0, \mc T],$$
with $g_\infty(0) = g$.
\et

\bdem
Consider the sequence $(M_i, g_i(t))$, with $t \in [0, \mc T]$, coming from Proposition \ref{LBT}.
Take some convex compact 
set $C_{j+1}= \beta^{-1}((-\infty,j+1])\subset M$
 from the convex 
exhaustion endowed with the Riemannian metric $g$. 
By the construction in Proposition \ref{IniAp}, we can view $C_{j+1}$ also 
as a subset of $M_i$ for $i\ge j+2$. 
Moreover, the metric $g_i(0)$ on $C_{j+1}$  converges to $g$ 
in the $C^\infty$ topology.
By Proposition~\ref{Int_BC} there is some constant 
$L_j$ with 
\begin{equation}\label{chen1}
 |\Rm_{g_i(t)}|\le \tfrac{L_j}{t} \qquad\mbox{ on  $B_{g_i(t)}(C_{j+1},1)$ \quad for all $t\in (0,\mc T]$
\quad and \quad $i\ge j+2$.}
\end{equation}

Since the metric $g_i(0)$ converges on $C_j$ in the $C^\infty$ 
topology to $g$ we can choose $\rho>0$ so small 
that 
\begin{equation}\label{chen2}
 |\Rm_{g_i(0)}|\le \rho^{-2 } \qquad \mbox{ on \quad  $C_{j+1}$ \qquad 
for \quad $i\ge j+2$.}
\end{equation}
After possibly decreasing $\rho$ we may assume 
that the $\rho$-neighborhood 
of $C_j$ is contained in $C_{j+1}$ with respect to the metric $g_i(0)$ for $i\ge j+2$. 
Combining the inequalities \eqref{chen1} and \eqref{chen2}
we are now in a position to apply  Theorem~\ref{SURF} in order to deduce that 
for some constant $\hat L_j>0$ we have 
\begin{equation}\label{chen3}
 |\Rm_{g_i(t)}|\le \hat{L}_j \quad \mbox{ on  \quad $C_{j}$ \quad for all \quad $t\in [0,\mc T]$ \quad 
and \quad $i\ge j+2$.}
\end{equation}

Combining this with an extension of Shi's estimate 
as stated in Theorem~\ref{ModShi}, we reach furthermore
\begin{equation*}
 |\nabla^k\Rm_{g_i(t)}|\le \hat{L}_{j,k} \quad \mbox{ on \quad  $C_{j}$ \quad for all \quad $t\in [0,\mc T]$ \quad 
and \quad $i\ge j+2$.}
\end{equation*}

From here,
 by standard arguments as in \cite{Hamcpt} (see Lemma 2.4 and remarks after it), we have that the metrics $g_i(t)$ on $C_j$ have all space and time derivatives uniformly bounded. Hence one can apply the Arzel\`a-Ascoli-Theorem to deduce that after passing to 
a subsequence $g_i(t)$ converges to $g_\infty(t)$ 
in the $C^\infty$ topology on $C_j\times [0,\mc T]\subset M \times \R$. 

Doing this for all $j\in \N$
 and applying the usual diagonal sequence argument we 
can, after passing to subsequence, assume that   $g_i(t)$ converges in the $C^\infty$ topology to 
a limit metric $g_\infty(t)$ on $C_j\times [0,\mc T]$ for all $j$.
By construction $g_{\infty}(t)$ 
is  a solution of the Ricci flow on $M$ 
with initial metric $g_\infty(0)=g$.
The completeness of $g_\infty(t)$ is a consequence of the following Lemma.
\edem

\begin{lem}\label{lem: complete1} There exists $L>0$ such that  
$B_{g_\infty(t)}(p_0,r)\subset B_{g_\infty(0)}(p_0,2r+L(t+1))$ 
for all positive $r$ and $t\in [0,\mc T]$.
\end{lem}
\begin{proof} This will follow by proving a uniform estimate for 
$(M_i,g_i,p_0)$.
Since $(M_i,g_i)$ is the double of a convex set, it has a natural $\Z_2$-symmetry which comes 
from switching the two copies of the double. 
As the Ricci flow on closed manifolds is unique, this symmetry 
is preserved by the Ricci flow. 
Thus the middle of $(M_i,g_i(t))$, being the fixed point set of an isometry,
 remains a totally geodesic 
hypersurface $N_i$. 
It is now fairly easy to estimate how the distance of $p_0$ 
to $N_i$ changes in time: 
Let $L_1$ be a bound on the eigenvalues of the Ricci 
curvature on $B_{g_i(t)}(p_0,1)$ for all $i$ and all $t\in [0,\mc T]$. 

If $c(s)$ is a minimal geodesic in $(M_i,g_i(t))$ 
from $p_0$ to $N_i$ then it follows 
for the left derivative of $r_i(t)=d_{g_i(t)}(p_0,N_i)$ that 
\begin{eqnarray*}
\tfrac{d}{dt} r_i(t)&\ge&- \int_{0}^{r_i(t)}\Ric_{g_i(t)}(\dc(s),\dc(s))\, ds \ge -L_1-\int_{1}^{r_i(t)}\Ric_{g_i(t)}(\dc(s),\dc(s))\, ds\\
&\ge& -L_1-(n-1)=-L_2,
\end{eqnarray*}
where we used the second variation formula in the last inequality.

If we put $D_i=d_{g_i(0)}(p_0,N_i)$ then we obtain 
$d_{g_i(t)}(p_0,N_i)\ge D_i-L_2 t$. Recalling that $d_{g_i(0)}(p,N_i) \geq d_{g_i(t)}(p,N_i)$, for any $r > 0$ we can find $i$ big enough so that
\begin{eqnarray*}
 B_{g_i(t)}(p_0,r)
&\subset& \{p\in C_i\cap M_i \mid d_{g_i(0)}(p,N_i)\ge D_i-L_2t-r\}.
\end{eqnarray*}
Next recall that $\beta$ is essentially a smoothing of a Busemann function 
and by assumption the level set  $\beta^{-1}(i)$ has 
Hausdorff distance at most $1$ to $b^{-1}(\ell_i)$ for a suitable 
$\ell_i$. 
Combining this with the previous inclusion and that we modified the metric 
in $C_i$ in a controlled way we deduce
\begin{eqnarray*}
B_{g_i(t)}(p_0,r)&\subset& \bigl\{p\in  b^{-1}((-\infty,\ell_i]) \mid 
d_{g_i(0)}(p,b^{-1}(\ell_i))\ge \ell_i-L_2t-r-L_3\bigr\}\\
&=&b^{-1}\bigl((-\infty, L_3+r+L_2t]\bigr)
\end{eqnarray*}
where $L_3=3+\diam_{g}( b^{-1}((-\infty,0]))$.
Finally, applying Lemma~\ref{Drees} 
gives
\[
b^{-1}\bigl((-\infty, L_3+r+L_2t]\bigr)\subset B_{g(0)}(p_0, 2r+L(1+t))
\] 
for a suitable large $L$.
\end{proof}

For some of the applications we will need a version 
of Lemma~\ref{lem: complete1} for abstract solutions 
of the Ricci flow with  $K^\C_{g(t)}\ge 0$.
\begin{lem}\label{lem: complete2}
 Let $(M,g(t))_{t\in [0,\mc T]}$ be a solution of the Ricci flow 
with $K^\C_{g(t)}\ge 0$. Suppose that 
$(M,g(t))$ is complete for $t\in [0, \mc T)$. 
If $p_0\in M$, then for some $C>0$
\[
 B_{g(t)}(p_0,R)\subset B_{g(0)}(p,R+C(1+t))\mbox{ for all $R\ge 0$, $t\in [0,\mc T]$} 
\]
In particular, $g(\mc T)$ is complete as well.
\end{lem}

\begin{proof}
There is nothing to prove in the compact case and thus 
we may assume that $(M,g(0))$ is open. 
After rescaling we may assume that the closure of $B_{g(\mc T)}(p_0,1)$ 
is compact and that $K_{g(t)}\le 1$ on $B_{g(t)}(p_0,1)$. 

We define $b_t\colon M\rightarrow \R$ 
by $b_t(q):=\limsup_{p\to\infty}\bigl(d_{g(t)}(p,p_0)-d_{g(t)}(p,q)\bigr)$. 
Notice that similarly to the Busemann function 
in subsection~\ref{subsec: chgr} $b_t$ is convex, proper and bounded below. 
In the sequel $\tfrac{d}{dt}$ refers to 
a right hand side Dini derivative. Similar to the proof of the previous lemma, it suffices to show\\[1ex]
{\bf Claim.} $ \tfrac{d}{dt}b_t(q)\ge - 4(n-1)$ for all $q\in M$.

Choose $q_k\to \infty$ with 
$b_t(q)=\lim_{k\to \infty}d_{g(t)}(q_k,p_0)-d_{g(t)}(q_k,q)$.
We may assume that the Busemann function $b_t$ is differentiable 
at $q_k$. 
Let $c_k$ (resp. $\gamma_k$) be a unit speed geodesic from 
$q_k$ to $p_0$ (resp. to $q$). 
We claim that the angle between 
$\dot{c}_k(0)$ and $\dot{\gamma}_k(0)$ converges to $0$.

Notice  that $b_t\ge \tilde b_t$, where 
$\tilde b_t$ is the Busemann function of $(M,g(t))$ defined with
respect to all rays emanating from $p_0$.
Therefore we deduce from Lemma~\ref{Drees} 
that $b_t(q_k)\ge (1-\delta_k)d(p_0,q_k)$
for some  sequence 
$\delta_k\to 0$.
After possibly adjusting $\delta_k$ we also may assume 
that $b_t(q_k)-b_t(q)\ge  (1-\delta_k)d(q,q_k)$. 
Since $s\mapsto b_t(c_k(s))$ and $s\mapsto b_t(\gamma_k(s))$ 
are convex $1$-Lipschitz functions, 
we get that $\ml \nabla b_t(q_k),\dc_k(0)\mr \le -(1-\delta_k)$ 
and $\ml \nabla b_t(q_k),\dgamma_k(0)\mr \le -(1-\delta_k)$.
Thus $\ml \dc_k(0),\dgamma_k(0)\mr\to 1$ as claimed.

By Toponogov triangle comparison theorem $d(c_k(1), \gamma_k(1))\to 0$.
If we put $\tq_k:=c_k(1)$, then 
$b_t(q)=\lim_{k\to \infty}d_{g(t)}(\tq_k,p_0)-d_{g(t)}(\tq_k,q)$.

We can now use the second variation formula combined 
with the curvature bounds on $B_{g(t)}(p_0,1)$ to see 
that $\tfrac{d}{dt}d_{g(t)}(\tq_k,p_0)\ge -4(n-1)$. 
Since $d_{g(t)}(\tq_k,q)$ is decreasing in $t$, we deduce
$\tfrac{d}{dt}\lim_{k\to \infty}d_{g(t)}(\tq_k,p_0)-d_{g(t)}(\tq_k,q)\ge -4(n-1)$. 
This in turn implies that the right derivative of $b_t(q)$ is bounded below 
by $-4(n-1)$ as claimed.
\end{proof}

\section{Miscellanea of auxiliary results for the general case}\lb{sec:misc}

\subsection{A splitting theorem for open manifolds with $K^{\ce} \geq 0$}

\bt
Let $(M^n, g)$ be an open, simply connected Riemannian manifold with $K^\ce_g \geq 0$. Then $M$ splits isometrically as  $\Sigma \times F$, where $\Sigma$ is the $k$-dimensional soul of $M$ and $F$ is diffeomorphic to $\re^{n - k}$. In particular, $F$ carries a complete metric of nonnegative complex sectional curvature.
\et

\bdem
By Theorem \ref{Strake-thm} due to M.~Strake, it is enough to show that the normal holonomy group of the soul $\Sigma$ is trivial, which in turn is equivalent (by a modification of \cite[Section 8.4]{Peters}) to prove
$$\<\Rm(e_1, e_2) v_1, v_2\> = 0 \quad \text{ for all } \quad e_1, e_2 \in T_p \Sigma \quad \text{ and all } \quad v_1, v_2 \in \nu_p \Sigma.$$

With such a goal, we look at the curvature tensor on the 4-dimensional space $N = {\rm span}\{e_1, e_2, v_1, v_2\}$, namely, we consider $\tilde \Rm = \Rm|_{\Lambda^2 N}$. It is well-known that one has the orthogonal decomposition $\Lambda^2 N = \Lambda^2_+ \oplus \Lambda^2_-$ into the eigenspaces of the Hodge star operator $\ast$ with eigenvalues $\pm 1$. This gives a block decomposition 
$$\tilde \Rm = \(\ba{cc} A &  B \\
^{t} B & C
\ea\)$$ 
with respect to the bases\begin{align*}
& \{b_1^\pm  = e_1\wedge e_2 \pm v_1\wedge v_2,\ b_2^\pm = e_1\wedge v_1 \mp e_2\wedge v_2, \ b_3^\pm = e_1 \wedge v_2 \pm e_2 \wedge v_1\} 
\end{align*}
for $\Lambda^2_\pm$.
By Cheeger and Gromoll  \cite[Theorem 3.1]{ChGr}
the mixed curvatures vanish, i.e. 
$\tilde \Rm(e_i\wedge v_j,e_i\wedge v_j)=0$. Thus
\begin{align*}
a_{22} + a_{33}+c_{22}+c_{33} & = 0.
\end{align*}

On the other hand, it is well 
known (cf.~\cite{Ham4pic}) that nonnegative isotropic curvature 
of $\tilde \Rm$ implies that 
the numbers $a_{22}+a_{33}$ and $c_{22}+c_{33}$ are nonnegative. 
Consequently,
\bec \lb{ACg0}
a_{22} + a_{33} = 0 = c_{22} + c_{33}.
\eec

Since $\tilde \Rm$ is a $4$-dimensional curvature 
operator with nonnegative sectional curvature, 
a result by Thorpe (see e.g.~\cite[Proposition 3.2]{Put}) ensures that 
we can find
a $\lambda \in \re$ such that 
$$\tilde \Rm+\lambda\(\ba{cc} I &  0 \\
 0 & -I
\ea\) = \(\ba{cc} A+\lambda I &  B \\
^{t} B & C-\lambda I
\ea\)$$ 
is a positive semidefinite matrix.
Combining this with $\eqref{ACg0}$, we obtain  $\lambda=0$, so
$\tilde \Rm$ itself is a nonnegative operator. Then 
$v_i\wedge e_j$ are in the kernel of $\tilde \Rm$.
Finally, the first Bianchi identity yields
$$\<\Rm(e_1 \wedge e_2), v_1\wedge v_2\> = - \<\Rm(v_1 \wedge e_1), e_2\wedge v_2\> - \<\Rm(e_2 \wedge v_1), e_1\wedge v_2\> = 0.$$
\edem

\subsection{Some preliminary estimates for Riccati operators}

In the space of self-adjoint endomorphisms $S(\re^n)$,  $A \leq B$ if $\<A v, v\> \leq \<B v, v\>$ for every $v \in \re^n$.

\begin{lem}\lb{Ric_comp_cons}
Let $A(s) \in S(\re^n)$ be a nonnegative solution of the Riccati equation
\bec \lb{Ric_abs}
A'(s) + A^2(s) + R(s) = 0, 
\eec
with $R(s) \geq 0$.  Assume also that $R$ and $|R'|$ are bounded for $s\in [0,1]$ by constants $C_R$ and $C_{R'} > 0$, respectively. Then there exists $A_0 \in S(\re^n)$ satisfying 
\bec \lb{A0_bdds}
A(0) \geq A_0, \qquad A_0 \leq C_R
\eec
and we can find an $\eps_0 = \eps_0(C_R, C_{R'}) > 0$ so that
\bec \lb{conv_ref}
\<A_0 w, w\>   \geq \eps_0 \<R(0) w, w\>^2 \qquad \text{for all} \quad w\in \re^n \quad \text{with} \quad |w|=1.
\eec
\end{lem}

\bdem
From \eqref{Ric_abs}, we can write
$$A(s) - A(0) = - \int_0^s \(A^2(\xi) + R(\xi)\)\, d \xi \leq  - \int_0^s  R(\xi)\, d \xi.$$
As this is valid for any $s$, using $A(s) \geq 0$ and $R(s) \geq 0$, we conclude
\begin{align*}
A(0) \geq \int_0^\infty R(\xi) \, d\xi \geq \int_0^1 R(\xi) \, d\xi =:A_0.
\end{align*} 
Clearly, $A_0 \leq C_R$. Next, take $w\in \re^n$ with $|w| = 1$, define $C:=\max\{C_R, C_{R'}\}$ and compute
$$\<R(s) w, w\> = \<\(R(0) + s R'(\xi)\)w, w\> \geq r(0) - C\, s,$$
where $r(0) = \<R(0) w, w\>$ and $r(0)/C \leq C_{R}/C \leq 1$. This allows to estimate
\begin{align*}
\<A_0 w, w\> & =\int_0^1 \<R(\xi) w, w\> d\xi \geq \int_0^1 \max\{0, r(0) - C\xi\} \, d \xi 
\\ & \geq \int_0^{r(0)/C} \(r(0) - C s\) \, ds = \frac{r(0)^2}{2 C},
\end{align*}
which gives the result taking $\eps_0 = \tfrac{1}{2 C}$.
\edem

\begin{lem}\lb{Ric_comp}
Let $A(s) \in S(\re^n)$ be a solution of  \eqref{Ric_abs}.  Suppose that $|R(s)| \leq C_R$ and $|R'(s)| \leq C_{R'}$ for small $s$. If there exists $\eps_0 > 0$ and $A_0\in S(\re^n)$ satisfying \eqref{A0_bdds} and \eqref{conv_ref}, then we can find an $s_0 = s_0(C_R)> 0$ such that 
\bec \lb{inA_bel}
A(s) \geq - C s^2 \,{\rm Id} \qquad \text{for all} \qquad s \in (0, s_0],
\eec
for some $C= C(C_R, C_{R'}) > 0$.
\end{lem}

\bdem
Using Riccati comparison (see e.g.~\cite{Esch78}) we can assume without loss of generality that 
$A(0) = A_0$. Next, we do a Taylor expansion for $A$ and use \eqref{Ric_abs} at $s = 0$:
\begin{align*} 
A(s) &\geq A(0) + s A'(0)  - C \, s^2\, {\rm Id} 
 \geq A(0) - s(A(0)^2 + R(0)) -  C \, s^2\, {\rm Id}
 \\ & \geq (1 - s C_R) A_0 - s R(0) - C s^2 \, {\rm Id},  
\end{align*}
which comes from \eqref{A0_bdds}. Notice that $C$ depends on $C_R$ and $C_{R'}$. Choose now any $w\in \re^n$ with $|w|=1$. Then, for $s\leq \tfrac1{2C_R} =: s_0$, our assumption \eqref{conv_ref} yields
\begin{align*}
\<A(s) w, w\> & \geq \frac1{2} \<A_0 w, w\> - s \<R(0) w, w\> - C s^2 \geq \frac1{2} \eps_0 r(0)^2 - s r(0) - C s^2 
\\ & = \(\sqrt{\frac{\eps_0}{2}} r(0) - \frac{s}{\sqrt{2 \eps_0}}\)^2 - \(C + \frac{1}{2 \eps_0}\) s^2 \geq -  \widetilde C(\eps_0, C)\, s^2.
\end{align*}
\edem

\subsection{Curvature estimates in terms of volume}

It will be useful for technical purposes to have in mind the following 
\begin{lem}\lb{BG_ref}
Let $(M^n, g)$ be a Riemannian manifold with $K_g \geq - \Lambda^2\ge -1$. 
Then there is a constant $C$ depending on $n$ such that
$$\frac{\vle_g \(B_g(p, R)\)}{\vle_g \(B_g(p, r)\)} \leq \(\frac{R}{r}\)^n (1 + C (\Lambda R)^2)$$
for all $p\in M$ and all $0 < r \leq R\le 2$.
\end{lem}

\bdem
Let us use the Taylor expansion for $\sinh \rho$, with $0 \leq \rho \leq 2$, to deduce
\begin{align*}
\rho^m \leq (\sinh \rho)^m = \bigg(\sum_{j = 0}^\infty \frac{\rho^{2 j + 1}}{(2 j + 1)!}\bigg)^m & \leq \rho^m (1 + \gorro C \rho^2)^m = \rho^m \sum_{j = 0}^m \bigg(\ba{c} \!\! m \\ \!j \ea\!\!\bigg) (\gorro C \rho^2)^j
\\ & \leq \rho^m (1 + C \rho^2).
\end{align*}

Now Bishop-Gromov's comparison theorem says that
\begin{align*}
\frac{\vle_g \(B_g(p, R)\)}{\vle_g \(B_g(p, r)\)} & \leq \ds \frac{\ds \int_0^R [\sinh(\Lambda \rho)]^{n -1} \, d \rho}{\ds\int_0^r [\sinh(\Lambda \rho)]^{n -1} \, d \rho} \leq \frac{\ds \int_0^R  (\Lambda \rho)^{n -1} (1 + C(\Lambda \rho)^2) \, d \rho}{\ds \int_0^r (\Lambda \rho)^{n -1} \, d \rho}
\\ & \leq \frac{(\Lambda R)^n/(\Lambda n) + C (\Lambda R)^{n + 2}/(\Lambda(n + 2))}{(\Lambda r)^n/(\Lambda n)} \leq \frac{R^n}{r^n} (1 + C (\Lambda R)^2).
\end{align*}
\edem

\begin{prop}\lb{curv_est_NN}
For any $\eps > 0$ we can find positive constants $\delta, \kappa, T$ such that if $(M^n, g(t))$ is a compact Ricci flow with 
\begin{eqnarray*}
K_{g(t)} \geq -\kappa \quad \text{ on } \, [0, \bar t] \qquad \mbox{ and } \qquad \frac{\vle_{g(0)} B_{g(0)}(\cdot, r)}{r^n} \geq (1 - \delta) \omega_n, 
\end{eqnarray*}
for some $r\in (0,1]$, then 
\bec \lb{ce_aevg}
|\Rm|_{g(t)} \leq \frac{\eps}{t} \qquad \text{on} \quad [0, \bar t] \cap [0, r^2 T(\eps)].
\eec
\end{prop}

\bdem
By rescaling it suffices to prove the statement for $r = 1$. Arguing by contradiction, suppose that there is an $\eps > 0$ and a sequence of Ricci flows $(M_i, g_i(t))$ defined on $[0, \bar t_i]$ satisfying
\bec \lb{asum_C}
K_{g_i(t)} \geq - \frac1{(n-1)i} \quad \text{ and } \quad \vle_{g_i(0)}B_{g_i(0)}(p, 1) \geq \(1 - \frac1{i}\) \omega_n
\eec
for all $p\in M_i$ and all $i$. But assume that we can also find a sequence of points and times $\{(p_i, t_i)\}$ such that
\bec \lb{asum_C2}
Q_i := |\Rm|_{g_i(t_i)}(p_i)  =\max_{q\in M_i}|\Rm|_{g_i(t_i)}(q) > \eps/t_i\, \qquad \mbox{ with \quad $t_i \to 0$.}
\eec

Next, we aim to show that the volume estimate in \eqref{asum_C} survives for some time. From \eqref{asum_C} and the evolution of $d_{g_i(t)}$ under \eqref{RF}, we deduce that
$$B_{g_i(t)}\big(p, e^{t/i}\big) \subset B_{g_i(t + \tau)}\big(p, e^{ (t + \tau)/i}\big) \qquad \text{ for all } \quad0\leq t < t + \tau \leq  \bar t_i.$$
Accordingly
\begin{eqnarray*}
\frac{\partial}{\partial t} \vle_{g_i(t)}\big(B_{g_i(t)}(p, e^{t/i})\big) 
&\geq&\frac{\partial}{\partial \tau}\bigg|_{\tau = 0}
 \vle_{g_i(t+\tau)}\big(B_{g_i(t)}(p, e^{t/i})\big)\\
&=& -\int_{B_{g_i(t)} \(p,\, e^{t/i}\)} \!\scal_{g_i(t)} \, d\mu_{g_i(t)}\\
&\geq& -C_n e^{(n-2)t/i},
\end{eqnarray*}
which follows from Petrunin's estimate (cf.~Theorem \ref{Pet}). This leads to 
\begin{align*}
\vle_{g_i(t)} \big(B_{g_i(t)}(p, e^{t/i})\big) \geq \vle_{g_i(0)} \big(B_{g_i(0)}(p, 1)\big) - C_n e^{(n-2) t/i} \, t \geq (1 - \eta_i) \, \omega_n
\end{align*} 
for all $t\in [0,t_i]$ and with $\eta_i\to 0$.
 Next, we can apply Lemma \ref{BG_ref} to conclude
\begin{align} \lb{cl_1} 
\vle_{g_i(t)} \(B_{g_i(t)} (p, r)\) & \geq
 \tfrac{1}{1 + C \frac{e^{2  t/i}}{i}}
\bigl(\tfrac{r}{e^{ t/i}}\bigr)^n
 \, \vle_{g_i(t)} \(B_{g_i(t)} (p, e^{ t/i})\) \nn 
 \\& \geq 
 \tfrac{e^{-n t_i/i}}{1 + C \frac{e^{2  t_i/i}}{i}}
 \, (1-\eta_i)r^n\omega_n \nn \\
& = (1-\mu_i)r^n\omega_n 
\end{align}
for all $0 < r \leq e^{t/i}$, $t\in [0,t_i]$ with $\mu_i\to 0$.

Now consider the rescaled solution to the Ricci flow $\bar g_i(t) = Q_i g(t_i + t/Q_i)$. Doing a time-picking argument, we can assume without loss of generality that
$$|\Rm|_{\bar g_i(t)} \leq 4 \quad \text{ on } B_{\bar g_i(0)} (p_i, 2)\quad \text{and} \quad t\in [-t_i Q_i/2, 0] \supset [-\eps/2, 0],$$
where the latter is true by \eqref{asum_C2}. In addition, 
\bec \lb{non_f}
|\Rm|_{\bar g_i(0)}(p_i) = 1.
\eec
A standard application of Shi's derivative estimates gives on $B_{\bar g_i(0)} (p_i, 1)$
\bec \lb{Cbd_bar2}
|\nabla^\ell \Rm|_{\bar g_i(t)} \leq \frac{C(n, \ell)}{(t + \eps/2)^{\ell/2}}; \quad \text{in particular} \quad |\nabla^\ell \Rm|_{\bar g_i(0)} \leq C(n, \ell, \eps).
\eec

After passing to a subsequence we may assume that $B_{\bar g_i(0)}(p_i,1)$ 
converges to a nonnegatively curved limit ball 
$B_{\bar g_\infty}(p_\infty,1)$ satisfying   \eqref{non_f} and 
\eqref{Cbd_bar2}. In particular $\vle_{\bar g_\infty}(B_{\bar g_\infty}(p_\infty,1))<\omega_n$. 
On the other hand, it is immediate 
from \eqref{cl_1} that $\vle_{\bar g_\infty}(B_{\bar g_\infty}(p_\infty,1))\ge \omega_n$ -- a contradiction.
\edem

\subsection{Smoothing $\bm{C^{1,1}}$ hypersurfaces}

\begin{lem}\label{lem: smoothing} Let $M^n$ be a smooth Riemannian manifold, $H$ 
a $C^{1,1}$ hypersurface, and $N$ a unit normal field.  
Suppose we have bounds $C_1\le A_H \le C_2$ on the principal curvatures of  
$H$ in the support sense. 
Then we can find a sequence of smooth hypersurfaces $H_i$ 
converging in the $C^1$-topology to $H$ such that   $C_1-\tfrac{1}{i}\le A_{H_i}\le C_2+\tfrac{1}{i}$.
If $H$ is invariant under the isometric action 
 of a compact Lie group $\lG$ on $M$,
 then one can assume in addition that $H_i$ is  
invariant under the action as well.
\end{lem}

\begin{proof} First, we give the proof in the case of a compact hypersurface.

We consider a small tubular neighborhood $U=B_{r_0}(H)$ of $H$. 
By assumption $U\setminus H$ has two components $U_+$ and $U_-$. 
We consider the function $f\colon U \rightarrow \R$ 
which is defined on $U_+\cup H$ as the distance to $H$ 
and on $U_-$ as minus the distance to $H$. 
Clearly $f$ is a $C^{1,1}$-function. 
Moreover, it is easy to deduce that for each $\eps$ we can find 
$r$ such that 
\[
C_1-\eps \le \nabla^2 f\le C_2+\eps \quad \mbox{ on \quad $B_r(H)$} 
\]
holds in the support sense.
Furthermore, we know of course $|\nabla f|\equiv 1$. 
Now it is not hard to see 
that for each $\eps>0$ we can 
find a smooth function $f_{\eps}$ 
on $B_{r/2}(H)$ satisfying 
\begin{eqnarray*}
C_1-2\eps \le \nabla^2 f_{\eps} &\le& C_2+2\eps\hspace*{1em} \mbox{ and} \\
|f_{\eps}-f|+|\nabla f-\nabla f_{\eps}|&\le& \min\{\eps,r/10\}\hspace*{3.4em}\mbox{ on $B_{r/2}(H)$.} 
\end{eqnarray*}
It is now straightforward to check that 
for $\eps_i=\tfrac{1}{10i(|C_1|+|C_2|+1)}$
we can 
put 
$H_i:=f_{\eps_i}^{-1}(0)$ and check the claimed bounds on 
the principal curvatures of $H_i$.

If $H$ is invariant under the isometric action 
 of a compact Lie group $\lG$, then $f(gp)=f(p)$ for all $p\in H$. 
By putting $\tf_{\eps}(p):= \tfrac{1}{\vol(\lG)}\int_{\lG} f_{\eps}(gp)\,d\mu(g)$
we obtain a $\lG$-invariant function with the same bounds on the 
Hessian as $f_\eps$. We can then define $H_i$ as before.

If the hypersurface is not compact one uses a ($\lG$-invariant) compact exhaustion and argues as before.
\end{proof}

\section{The general case}\lb{sec:gral}
\subsection{Estimates on the Hessian of the squared distance function}

\begin{prop}\lb{AR32}
Let $(M^n, g)$ be an open manifold  with $K_g \geq 0$, let $C_\ell$ be 
a sublevel set of the Busemann function (see subsection~\ref{subsec: chgr}), and $p\in \partial C_\ell$.  
For each unit normal vector $v \in N_p C_\ell$ there
is a smooth hypersurface $S$ supporting $\partial C_\ell$ at $p$
from the outside such that $T_pS$ is given by the orthogonal complement of 
$v$, and the second fundamental form $A$ of $S$ satisfies
\bec \lb{2ff-ss}
u\geq \<A_v w, w\>  \geq c \, \Rm(w, v, v, w)^2 \quad \text{ for all} \quad    w \in T_p S \quad \text{with} \quad |w| = 1,
\eec
for some positive constants $c$ and $u$ 
depending on $C_\ell$. 
\end{prop}

\bdem We fix $r>0$ smaller than a quarter of the convexity radius of 
$C_{\ell + r}$. Proposition \ref{Yim} by Yim ensures that any element of $N_p C_{\ell}$ can be obtained, up to scaling, as (hereafter we use Einstein sum convention)
$$ \alpha^i u_i \qquad \text{ with } \quad |u_i| = 1, \quad  \alpha_0 + \ldots + \alpha_k =1 \quad \text{and} \quad \alpha_i \geq 0,$$
where each $u_i \in {\rm span}(T_p C_{\ell})$ is such that $\gamma_i(s) = \exp_p(s u_i)$ is the minimal geodesic from $p$ to a point of $\partial C_{\ell+2r}$. As ${\rm dim} (C_{\ell}) = n$, we can choose $k\le n$ by Carath\'eodory's Theorem (cf.~\cite{Rocka}).

Consider $q_i:=\gamma_i(r)\in C_{\ell+r}$ 
and the hyperplane $V_i\subset T_{q_i}M$ perpendicular to
$\gamma_i'(r)$. Since $\gamma_i'(r)\in N_{q_i}C_{\ell+r}$ 
it follows that $H_i:=\exp(B_r(0)\cap V_i)$ is a smooth hypersurface 
supporting $C_{\ell+r}$ from the outside. 

Then $\varphi_{i}:=r+\ell-d(H_i,\cdot)$ is a lower 
support function of the Busemann function $b$ at 
$p$ (which can be seen using e.g.~Lemma \ref{wu_le} by Wu).
 Note that 
$\nabla^2\varphi_i|_{\gamma_i(s)}$ 
is a positive semidefinite solution of a Riccati equation for
$s\in [0,r]$. So we clearly have upper bounds (just depending on $C_\ell$) for $\nabla^2\varphi_i|_p$.
 Lemma~\ref{Ric_comp_cons} now yields \begin{equation}
\lb{DL1}\nabla^2\varphi_i|_p(w,w)\ge \eps_0 \Rm(w,u_i,u_i,w)^2
     \end{equation}
for all unit vectors $w\in T_pM$.

As mentioned above there is some $\lambda>0$ such that 
$\lambda v=\alpha^iu_i$ with $\sum_i\alpha_i=1$.
Define $\phi =  \alpha^i \varphi_{i}$, which is a function whose gradient is $\lambda  v$.
 Since it is a convex combination of lower support functions for $b$ at $p$,
 $\phi$ is also a lower support function for $b$ at $p$; therefore, $b^{-1}\((-\infty, \ell]\) \cap B_r(p) \subset \phi^{-1}\((-\infty, \ell]\) \cap B_r(p)$. Consequently, if we define $S$ as the level set $\phi^{-1}(\ell)\cap B_r(p)$, then $T_p S$ is orthogonal to $v$ and $S$ supports $C_\ell$ at $p$ from the outside. Moreover, the second fundamental form of $S$ at $p$ is proportional to $\nabla^2 \phi |_p$, and from \eqref{DL1} we have
\bec \lb{DL2}
\nabla^2 \phi |_p (w, w) =  \alpha^i \nabla^2 \varphi_{ i}|_p (w, w) \geq  \eps_0 \, \alpha^i  R_w(u_i, u_i)^2.
\eec

Next, using $K_g(\alpha_i u_i - \alpha_j u_j, w) \geq 0$ we can  estimate the curvature:
$$R_w(\alpha^i u_i,  \alpha^j u_j)    \leq  
\frac1{2} \sum_{i,j} \(\alpha_i^2 R_w(u_i, u_i) + \alpha_j^2 R_w(u_j, u_j)\) \leq (n + 1)  (\alpha^i)^2 R_w(u_i, u_i).$$
Now, combining a discrete version of H\"older's inequality applied to \eqref{DL2}, that $\alpha_i \leq 1$ and the above computation, we reach
\begin{align*}
\nabla^2 \phi|_p(w,w) & \geq \frac{\eps_0}{n+ 1} \Big( \sum_i \, \sqrt{\alpha_i}\, R_w (u_i, u_i)\Big)^2 \geq \frac{\eps_0}{n+ 1} \Big( \sum_i \, \alpha_i^2\, R_w (u_i, u_i)\Big)^2 
\\ & \geq \frac{\eps_0}{(n + 1)^2} R_w(\alpha^i u_i, \alpha^j u_j)^2 = c(n, \eps_0) R_{\alpha^i u_i}(w, w)^2.
\end{align*}

Finally, the statement follows since the second fundamental form of $S$ satisfies $\<A_v w, w\> \geq c(n, \eps_0) \lambda^3 R_v(w, w)^2$,  and it is easy to see that $\lambda$ is bounded below by a constant just depending 
on $C_{\ell}$. 
\edem

\begin{cor} \lb{est_Hesf}
Consider $(M^n, g)$ and $C= C_\ell$ as in Proposition \ref{AR32}. Then there exists a neighborhood $U$ of $C_\ell$ such that $f = d^2(\cdot, {C})$ is a $C^{1,1}$ function on $U$ and satisfies the following estimates
\bec \lb{in_f}
- \lambda f^{3/2} \leq \nabla^2 f \leq 2 \qquad \text{on} \qquad U 
\eec
in the support sense, for some positive constant $\lambda = \lambda(C)$. 
\end{cor}

\bdem 
By a result of Walter (see Theorem \ref{tubular}) 
we can find
 a tubular neighborhood $U$ of $C$, 
such that $f$ is $C^{1,1}$ on $U$. 

Let $q\in U$ and let $p\in \partial C$ denote a point 
with $d(q,p)=d(q,C)$. 
Clearly $d(\cdot,p)^2$ is an upper support function of $f$ at $q$ 
and thus $\nabla^2 f|_q \leq 2$.

In order to get the lower bound, we consider 
a minimal unit speed geodesic $\gamma(s)$ 
from $p$ to $q$. 
The initial direction $v =\gamma'(0)$ 
is a normal vector, and by Proposition~\ref{AR32}
we can find a hypersurface $S$ touching $C$ from the outside 
at $p$ such that $T_pS$ is normal to $v$,
and the second fundamental form of $S$ 
is bounded  by $u\ge \ml A_v w,w\mr\ge c \,\Rm(v,w,w,v)^2$ 
for any unit vector $w$.

Notice that $a^2:=d(S,\cdot)^2$ is a lower 
support function of $f$ at $q$.
Since $A(s) = \nabla^2 a|_{\gamma(s)}$ satisfies 
a Riccati equation with $A(0)=A_v$, 
we can employ Lemma \ref{Ric_comp} (for which we can take $A_0 = A_v$, since the latter is bounded above) to conclude $A(s) \geq - C s^2$. 
Consequently, $\nabla^2 f|_q\ge -2C \, a(q)\,  d(p,q)^2=-2Cf^{3/2}(q)$.
\edem

\subsection{A sequence of graphical sets with controlled curvatures}

For any $r >0$ and any set $S \subset M$ consider the 
tubular neighborhood $B_r(S) = \cup_{p \in S} \overline{B}_r(p)$. 
\begin{prop}\lb{2ff_bdd}\lb{smoothing}
Let $C \subset (M^n, g)$ be a sublevel set of the Busemann function
 (see subsection~\ref{subsec: chgr}). Then we can construct a sequence $\{D_k\}_{k = 1}^\infty$ of $C^{\infty}$ closed hypersurfaces of $B_1(C) \times [0, 1]$ which converges in the Gromov-Hausdorff sense to the double of $C$, and whose principal curvatures $\lambda_i$ satisfy
\bec \lb{ppal_est}
- \frac{b}{k^2} \leq \lambda_i \leq B \, k
\eec
for all $1\leq i \leq n$ and some positive constants $b, B$
depending on $C$.
Hence, if we endow $D_k$ with 
the induced Riemannian metric $g_k$, we get the curvature estimates
\bec \lb{conCSC}
 -\frac{\tb}{k}  \leq K^{\ce}_{g_k}\,\, \quad \mbox{ and }\,\, \quad 
|\Rm_{g_k}| \leq  \tilde B k^2\qquad \text{on} \qquad  D_k.
\eec

\end{prop}

\bdem In a first important step 
we will construct a closed  $C^{1,1}$ hypersurface $D_k$ so that 
 \eqref{ppal_est} holds for its principal curvatures
in the support sense.
Define $\phi_k = \frac1{k} \, \phi(k^2 f)$, where as before $f = d^2(\ccdot,  C)$ and $\phi: [0, 1] \fle [0,1]$ is a smooth function satisfying
\bi
\item[(a)] $\phi \equiv 0$  on $[0, 1/4]$ and $\phi(1) =1$;

\item[(b)] on $(1/4, 1)$: $\phi', \phi''$ are positive and $\phi'' \leq \alpha (\phi')^3$ for some finite $\alpha >0$;

\item[(c)] $\phi^{-1}$ has all left derivatives vanishing at $1$.
\ei
Notice that (c) implies that $\phi'$ and $\phi''$ tend to infinity at $1$. Hereafter, $\phi, \phi'$ and $\phi''$ will always be evaluated at $k^2 f$ (without saying it explicitly).

Consider the tubular neighborhood $U$ from Corollary \ref{est_Hesf}, and take
$$G_k = \{(p, \phi_k(p))\, : \, p \in B_{1/k}(C) \cap U\}$$
which is a hypersurface of the cylinder $B_1(C) \times [0, 1/k]$. Observe that by (a) $G_k$ can be written as the union of $\(B_{1/(2 k)}(C) \cap U\) \times \{0\}$ (which is totally geodesic in the cylinder) and the graphical annulus 
\bec \lb{def_Ak}
A_k = \left\{(p, \phi_k(p)) \, :\, p\in U \, \mbox{ and } \,  \tfrac{1}{2k} \leq d(p, C) \leq \tfrac{1}{k}\right\}
\eec
whose second fundamental form $h$ is given by
$$h = \frac{\nabla^2 \phi_k}{\sqrt{1 + |\nabla \phi_k|^2}}, \quad \text{where } \ba{l} \nabla \phi_k = k \,\phi'  \nabla f = 2 k \phi' d \nabla d \medskip \\ \nabla^2 \phi_k = k^3 \phi'' \nabla f \otimes \nabla f + k \, \phi' \nabla^2 f\ea$$

We need to estimate the principal curvatures of $A_k$ to prove \eqref{ppal_est}. With such a goal, take $e_1 = \frac{\nabla d}{\sqrt{1 + \<\nabla d, \nabla \phi_k\>^2}}$ and complete to form a basis $\{e_i\}$ orthonormal with respect to the metric induced on the graph $\tilde g = g + \nabla \phi_k \otimes \nabla \phi_k$, and which diagonalizes $h$.
Notice that
\begin{align*}
 \frac{k^3 \phi'' \<\nabla f, e_1\>^2}{\sqrt{1 + (2\,k\,  \phi' \, d)^2}}
\leq \frac{k^2}{2 d} \frac{\phi''}{\phi'}  \frac{\<\nabla f, \nabla d\>^2}{1 + \<\nabla d, \nabla \phi_k\>^2}
  \leq \frac{k^2}{2 d} \frac{\phi''}{(\phi')^3}  \frac1{k^2} \leq \frac{\alpha}{2 d} 
\end{align*}
and
$$\Lambda := \frac{k\, \phi' \nabla^2 f(e_i, e_i)}{\sqrt{1 + (2\,k\,  \phi' \, d)^2}} \leq \frac1{2 d}\nabla^2 f(e_i, e_i) \leq \frac{1}{d},$$
which comes from \eqref{in_f}.
Therefore, from \eqref{def_Ak} we obtain  $\lambda_i \leq \alpha \, k + 2 k =: B k.$

On the other hand, using $\phi'' \geq 0$,  \eqref{in_f} and \eqref{def_Ak}, we have
$$\lambda_i = h(e_i, e_i)  \geq \Lambda  \geq - \lambda \, f^{3/2} \frac1{2 d} = - \frac{\lambda}{2} d^2 \geq -\frac{\lambda}{2 k^2}.$$

All the previous computations are true at almost every point of $A_k$, since Corollary \ref{est_Hesf} ensures that $f$ is $C^{1,1}$ 
on $U$ and thus twice differentiable almost everywhere. At the remaining points all the above estimates are still valid in the support sense (just redo the proof substituting  $f$ by its support functions).

Clearly the hypersurfaces $D_k=D(G_k)$ converge in the Gromov-Hausdorff sense 
to the double 
$D(C)$ of the convex set $C$. 
Employing Lemma~\ref{lem: smoothing}, 
after increasing $b$ and $B$ slightly, we can find a smooth 
hypersurface $\tD_k$ which is
$C^1$ close to $D_k$ such that 
the estimate ~\eqref{ppal_est} remains valid. 
Clearly we can assume that $\tD_k$ still converges to 
$D(C)$. Finally, rename $D_k = \gorro D_k$, and notice that \eqref{conCSC} now 
 follows from ~\eqref{ppal_est} and the Gau{\ss} equations. 
\edem

Observe that each $(D_k, g_k)$ constructed above is not anymore nonnegatively curved, but we have a precise control of its curvature given by \eqref{conCSC}. Using the short time existence theory from \cite{ham3D}, we have the following immediate
\bco \lb{RFsmoothing}
There exists $T_k > 0$ such that $(D_k, g_k(t))$ is a sequence of solutions to the Ricci flow for $t\in [0, T_k)$ starting at the smooth closed manifolds $(D_k, g_k)$ from Proposition \ref{smoothing}.
\eco

\subsection{Curvature estimates for the Ricci flow of our initial sequence of smoothings} \lb{curv_est_sec}
We consider a fixed convex exhaustion $C_\ell=b^{-1}((-\infty,\ell])$ 
as in subsection~\ref{subsec: chgr}. 
For each $C_\ell$ we apply Proposition \ref{smoothing} with $C=C_\ell$, 
 let $( D_{\ell, k}, g_{\ell, k}(t))$ denote the Ricci flow from Corollary \ref{RFsmoothing} and put $g_{\ell, k}=g_{\ell, k}(0)$. Moreover, when a constant $B$ depends on $C_\ell$ we will write $B_\ell$ to denote $B(C_\ell)$. 
 
 Our next concern is to extend the curvature estimates in \eqref{conCSC} at least for a short interval of time, where the important point is that the length of such an interval is independent of $k$.
It is somewhat surprising that we can only prove this
if $\ell$ is large enough. Ultimately this in turn
 is due to the following

\begin{lem}\lb{GK_ap}
Let $( D_{\ell, k}, g_{\ell, k})$ be the closed smooth manifolds constructed in Proposition \ref{smoothing}, and take $p\in  D_{\ell, k}$. Then we can find $r = r(\ell) \in (0, 1]$ (possibly converging to $0$ with $\ell \to \infty$) 
and $\eta_\ell\to 0$  independent of $k$ such that
\bec \lb{Vol_AE}
\frac{\vle_{g_{\ell, k}} \!\(B_{g_{\ell, k}}(p, r)\)}{r^n} \geq (1 - \eta_\ell) \,\omega_n\,\,\, \quad \mbox{ for all $k$.}
\eec
\end{lem}

\bdem As the manifolds $( D_{\ell, k}, g_{\ell, k})$ converge 
to the double $D(C_\ell)$ of $C_\ell$ in the Gromov-Hausdorff sense,
the continuity of volumes (see e.g.~\cite[Theorem 5.9]{CheCol} by Cheeger and Colding) gives 
$\lim_{k\to \infty} \vle_{g_{\ell, k}}\! \(B_{g_{\ell, k}} (p_k, r)\) = \mc H^n_{D(C_\ell)}(B(p_\infty, r))$. 
Thus it suffices to prove that small 
balls in $D(C_\ell)$ have nearly Euclidean volume provided that $\ell$ 
is large.

This essentially follows
from Lemma \ref{GuiKa} by  Guijarro and  Kapovitch
which ensures that,  for $p\in \partial C_{\ell}$ with large $\ell$,  $T_p C_\ell$ is close to a half-space, and so $\vle \{v \in T_p C_\ell : |v| < r\} = \frac{1}{2} (\omega_n -  \eps_{\! \ell}) r^n$ with $\eps_\ell \to 0$ as $\ell \to \infty$.
 As $C_\ell$ is a convex set in a Riemannian manifold, we can find 
for each $p \in D(C_\ell)$ a number $r(p)$ small enough
 so that the volume of a geodesic ball $B(p, r)$ in $D(C_\ell)$ is 
$\ge (\omega_n -  2\eps_{\! \ell}) r(p)^n$. 

To remove the dependence on $p$,  we choose 
a finite subcover $\bigcup_{i = 1}^k B(p_i, \eps_{\! \ell} \, r_i)$ of $\bigcup_{p\in D(C_\ell)}B(p,\eps_{\! \ell} \, r(p))$, where $r_i = r(p_i)$ and we take $r_0 = \min_i r_i$. Then any $q \in D(C_\ell)$ is contained in $B(p_i, \eps_{\! \ell} \, r_i)$ for some $i$. Notice that $B(p_i, r_i) \subset B(q, (1 + \eps_{\! \ell}) r_i)$ and thus $\vle \(B(q, (1 + \eps_{\! \ell}) r_i)\) \geq (\omega_n -  2\eps_{\! \ell}) r_i^n$. Finally, apply volume comparison to get $\frac{\vle \(B(q, r_0)\)}{r_0^n} \geq \frac{\omega_n -  2\eps_{\! \ell}}{(1 + \eps_{\!\ell})^n}$. 
\edem

\bp \lb{prop:curv_est}
There exists some $\ell_0 >0$ and for each $\ell\geq \ell_0$ 
exists a time $T_\ell > 0$ (independent of $k$) such that for the Ricci flow $(D_{\ell, k}, g_{\ell, k}(t))$ constructed in Corollary \ref{RFsmoothing} we have 
\bec \lb{curv_est}
K^\ce_{g_{\ell, k}(t)} \geq - \frac{1}{\sqrt{k}} \qquad \text{and} \qquad |\Rm|_{g_{\ell, k}(t)} \leq \frac1{t}
\eec
for all  $t\in (0, T_\ell]$ 
and all sufficiently large $k$.
\end{prop}

\bdem
Unless otherwise stated, all the curvature quantities hereafter correspond to $g_{\ell, k}(t)$. We consider 
a maximal solution $(D_{\ell, k}, g_{\ell, k}(t))$ of the Ricci 
flow with $t\in [0,T_{\ell, k})$.
By \eqref{conCSC} there is some constant $B_\ell$ such that 
\begin{equation}\label{initial}
K^\C(0)\ge -\tfrac{B_{\ell}}{k},\qquad \mbox{ and } \qquad  |\Rm(0)|\le B_{\ell}\,k^2.
\end{equation}
Henceforth we will restrict our attention to $k\ge 4B_{\ell}^2$.

We define $t_{\ell, k}$ as the minimal time for which we can find  some complex plane $\sigma$ in $T^\C D_{\ell, k}$ with
 $K^{\C}(t_{\ell, k})(\sigma)=-\tfrac{1}{\sqrt{k}}$.
If such a time does not exist, we put $t_{\ell, k}=T_{\ell, k}$. In particular, we have for the usual sectional curvature
\begin{eqnarray}\label{lower sec}
 K(t)\ge -\tfrac{1}{\sqrt{k}} \qquad \mbox{ for all \qquad $t\in [0,t_{\ell, k}]$.}
\end{eqnarray}

We put 
\[
 u(t):= 4n(|\Rm(t)|+1).
\]
Using the initial estimate~\eqref{initial}
it is not hard to obtain a doubling estimate 
for $u(t)$. In fact, the application of \cite[Lemma 6.1]{CLN} gives 
\begin{equation}\label{u doubling}
 u(t)\le L_\ell \, k^2 \qquad \mbox{ for all \qquad $t\in \bigl[0,\tfrac{1}{L_{\ell}k^2}\bigr]$ }
\end{equation}
for some positive constant $L_\ell$.

By Lemma~\ref{GK_ap}, for any $\delta>0$ we can find  
an $\ell_0$ such that for each $\ell\ge \ell_0$ 
there is $r=r(\ell)$ with
$\vol_{g(0)}(B_{g(0)}(p,r))\ge (1-\delta)\, r^n \omega_n$. 
Combining this with   \eqref{lower sec} and 
Proposition~\ref{curv_est_NN} we deduce that, for $\ell\ge \ell_0$,
 there is some 
$\bar t_{\ell}$ and
$k_0=k_0(\ell)$ such that 
\begin{equation}\label{decay u}
 u(t)\le \frac{1}{10 t} \quad \mbox{ for all \quad $t\in [0,t_{\ell, k}]\cap [0,\bar t_\ell]$ \quad and all \quad $k\ge k_0$.}
\end{equation}
Thus the inequalities~\eqref{curv_est} hold 
for $t\in [0,t_{\ell, k}]\cap [0,\bar t_\ell]$
and  it suffices to 
check that $t_{\ell, k}$ is bounded away from $0$ for $k\to \infty$.
In particular, it is enough to consider hereafter $k\ge k_0$ 
with $t_{\ell, k}<\min\{T_{\ell, k},\bar{t_\ell}\}$.

In order to get a lower bound on $t_{\ell, k}$
we have to estimate $K^\C(t)$ from below.
Consider the algebraic curvature operator $\widetilde \Rm := \Rm + \lambda(t) I$, where $I_{ijkl} = \delta_{ik} \delta_{jl} - \delta_{il} \delta_{jk}$ represents the curvature operator of the standard  unit sphere, and $\lambda(t) \geq 0$. Under the Ricci flow, $\widetilde \Rm$ evolves according to
$$\(\parcial{}{t} - \Delta\) \widetilde \Rm = \lambda'(t) I + 2(\Rm^2 + \Rm^\sharp).$$
Next, recall the formula (cf. \cite[Lemma 2.1]{BW})
$$(\Rm + \lambda I)^2 + (\Rm + \lambda I)^\sharp = \Rm^2 + \Rm^\sharp + 2 \lambda {\rm Ric} \wedge {\rm id} + \lambda^2 (n - 1) \, I.
$$
It is easy to see that
$$2{\rm Ric} \wedge {\rm id} + \lambda (n - 1) \, I\le \tfrac{u(t)}{2} I
$$
holds provided that $\lambda\le 1$.  We reach
\begin{align*}
\(\parcial{}{t} - \Delta\) \widetilde \Rm & \geq 2 \big(\widetilde \Rm^2 + \widetilde \Rm^\sharp\big)
  + \big[\l'(t) - \l(t) u(t)\!\big] I.
\end{align*}
As the ODE $\Rm'=\Rm^2+\Rm^\#$ preserves $K^\ce \geq 0$,
we can use the maximum principle to ensure that: 
If we define $\lambda(t)$  as the solution of the initial value problem
$$\ba{l}\l'(t) = \l(t) u(t) 
\medskip \\  \l(0) = \tfrac{B_\ell}{k} \quad (\text{from } \eqref{initial}) \ea,$$
it follows that $\widetilde \Rm(t)$ has nonnegative complex sectional curvature for
$t\in [0,t_{\ell,k}]$.
Hence 
$$K^\ce(t) \geq -\tfrac{B_\ell}{k} e^{\int_0^tu(\tau) d\tau} 
\hspace*{1em}\mbox{ for all $t\in [0, t_{\ell, k}]$.}$$
Combining this with \eqref{u doubling} and \eqref{decay u} we 
deduce
\begin{eqnarray*}
K^\ce(t)&\geq& -e \,\tfrac{B_\ell}{k}\hspace*{5.5em} \mbox{ for $t\in \bigl[0,\tfrac{1}{k^2L_\ell}\bigr]\cap [0,t_{\ell, k}]$, \quad  and}\\
K^\ce(t)&\geq& -e \, \tfrac{B_\ell}{k}\, \bigl(L_\ell k^2 t\bigr)^{\frac{1}{10}}
\hspace*{1em} \mbox{ for $t\in \bigl[\tfrac{1}{k^2L_\ell},t_{\ell, k}\bigr].$}
\end{eqnarray*}
Using that by construction the minimum of 
$K^\ce(t_{\ell, k})$ is given by $-\tfrac{1}{\sqrt{k}}$, we obtain
 the desired uniform lower bound on $t_{\ell, k}$. 

\edem

\subsection{Reduction to the positively curved case}
We have the following improvement of Proposition~\ref{IniAp}.
\begin{prop}\lb{reduction}
Let $(M^n, g)$ be an open manifold with $K^\ce_g \geq 0$ 
whose soul is a point $p_0$. 
Then there is a sequence of closed manifolds $(M_i,g_i,p_0)$
with $K^\ce_{g_i} > 0$ 
converging in the Cheeger-Gromov sense 
to $(M,g,p_0)$.
\end{prop}

\bdem
Consider the sets $C_\ell =\{b\leq \ell\}$ from 
subsection~\ref{subsec: chgr}.
 Summing up, from Proposition \ref{smoothing} and Corollary \ref{RFsmoothing} we have  a sequence $(D_{\ell, k}, g_{\ell, k}(t))$ of Ricci flows satisfying \eqref{curv_est} on $(0, T_\ell]$ for all $\ell \geq \ell_0$.
Using Petrunin's result (Theorem \ref{Pet}) similar to the proof of Proposition \ref{curv_est_NN}, we see that the volume estimate \eqref{Vol_AE} 
 remains valid for $(D_{\ell, k}, g_{\ell, k}(t))$ 
provided we double the constant $\eta_\ell$ and we 
assume $t\in [0,t_\ell]$. 
This, combined with \eqref{curv_est}, allows to apply  Theorem \ref{ChGrT} by Cheeger, Gromov and Taylor to reach a uniform lower bound for 
the injectivity radius 
${\rm inj}_{g_{\ell,k}(\bar t)}\ge c(\ell, \bar t)$, for any $\bar t \in (0, t_\ell]$. 

Then, we can apply Hamilton's compactness (Theorem \ref{locHCT}) to get a compact limiting Ricci flow 
$(D_{\ell, \infty}, g_{\ell, \infty}(t))$ on $(0, T_\ell]$
 with $K^\ce_{g_{\ell, \infty}(t)} \geq 0$. Arguing e.g.~as in \cite[Theorem 9.2]{Simon}, we deduce that $(D_{\ell, \infty}, d_{g_{\ell, \infty}(t)})$ converges (in the Gromov-Hausdorff sense as $t \to 0$) to $(D(C_\ell), d_{g_\ell})$; in particular, $ D_{\ell, \infty}$ is homeomorphic to 
the sphere $D(C_\ell)$. By Proposition~\ref{hol} 
$ K^\ce_{g_{\ell, \infty}(t)} > 0$ for all $t\in (0,T_{\ell}]$.

On the other hand, for any $\eps>0$ we can view $C_{\ell-\eps}$ 
as a subset of $D_{\ell, k}$ for all $k\in \N\cup \infty$. 
Combining Theorem~\ref{SURF}  and a generalization of Shi's estimate
(see
Theorem~\ref{ModShi}) with \eqref{curv_est}
we see that on $C_{\ell-\eps}$ the metric $g_{\ell, \infty}(t)$ 
converges for $t\to 0$ in the $C^\infty$ topology 
to $g$. We choose $t_\ell$ 
so close to zero that $(C_{\ell-\eps}, g_{\ell, \infty}(t_\ell))$ 
converges in the $C^\infty$ topology to $(M,g,p_0)$
as $\ell\to \infty$.
\edem

If the soul of the manifold of Theorem \ref{mainT} is a point,
one can now deduce the conclusion of Theorem \ref{mainT} 
completely analogously to Section~\ref{sec:pos-case}
using  Proposition~\ref{reduction} in place of Proposition~\ref{IniAp}.

\begin{proof}[\bf Proof of Theorem \ref{mainT}] If $M$ is not simply connected, we consider its universal cover $\gorro M$. 
The goal is to construct a Ricci flow $(\gorro M,g(t))$ 
on $\gorro M$ for which each $g(t)$ is invariant under 
$\Iso(\gorro M,g(0))$. 
By Theorem \ref{splitting}, $\gorro M$ splits isometrically as $\Sigma^k \times F$, where $\Sigma$ is closed and $F$ is diffeomorphic to $\re^{n - k}$ with $K^\ce_{g_F} \geq 0$. 
By \cite[Corollary 6.2]{ChGr} of Cheeger and Gromoll 
$F$ splits isometrically as 
$F=\R^q\times F'$ where $\R^q$ is flat and $F'$ has a compact isometry group. 
Clearly there is a Ricci flow 
 $(\R^q\times \Sigma,g(t))$ which is invariant under $\Iso(\R^q\times \Sigma)$ 
and thus it suffices to find a Ricci flow $(F',g(t))$  
 which is invariant under $\Iso(F')$.  
Using \cite[Corollary 6.3]{ChGr} by Cheeger and Gromoll, we can find 
$o\in F'$ which is a fixed point of $\Iso(F')$.
We now define the Busemann function 
on $F'$ with respect to this base point. 
Then all sublevel sets $C_{\ell}$, the doubles $D(C_\ell)$ and 
the smoothings of the double $D_{\ell,k}$ 
come with a natural isometric action of $\Iso(F')$. 

Since the Ricci flow on compact Riemannian manifolds is 
unique, the Ricci flow $(D_{\ell,k},g_{\ell,k}(t))$ 
is invariant under $\Iso(F')$; hence the same 
holds for the limit $(D_{\ell,\infty},g_{\ell,\infty}(t))$,
and finally for the limiting Ricci flow on $F'$. 

In summary, there is a Ricci flow  $(\gorro M,g(t))$ 
 with $K^\ce_{g(t)} \geq 0$ which 
is invariant under $\Iso(\gorro M,g(0))$ and so descends 
to a solution on $M$.
\end{proof}

\section{Applications}\label{sec: applications}
 \subsection{Proof of Corollary \ref{cor: noncollapsed}} 
Arguing as before, it is enough to consider the case where the soul is a point. Redoing the arguments from the proof of Proposition \ref{LBT} and using \eqref{non_col_hyp}, we deduce that our Ricci flow exists until time $\mc T = \frac{v_0}{2 C_n}$. Plugging this and \eqref{non_col_hyp} into a reasoning like in Corollary \ref{Cor_vol}, we reach 
\bec \lb{vol_est_cor3}
\frac{\vle_{g(t)} \(B_{g(t)}(p, r)\)}{r^n} \geq \frac{v_0}{2} > 0  \qquad \text{for } r\in (0, 1], \quad p\in M, \quad t\in [0, \mc T].
\eec
 
 Now assume that the claim about bounded curvature does not hold, i.e., there exists a sequence of Ricci flows $(M_i, g_i(t))$ constructed as in Theorem \ref{mainT} (in particular, $K^\ce_{g_i(t)} \geq 0$ and $(M_i, g_i(t))$ satisfies a trace Harnack inequality) and points $(p_i, t_i) \in M_i \times (0, \mc T)$ with $\scal_{g_i(t_i)} > 4^i/t_i$.  By means of a point picking argument as in the proof of Proposition \ref{Int_BC} on the relatively compact set $B_{g_i(t_i)}(p_i, 1)$, we get a sequence of points $\{\bar p_i\}_{i\geq i_0}$ such that, after parabolic rescaling of the metric with factor $Q_i= \scal_{g_i(t_i)}(\bar p_i)$, we get for the rescaled metric $\tilde g_i(s)$
$$|\Rm|_{\tilde g_i(s)} \leq 8 \qquad \text{on } \quad  B_{\tilde g_i(0)}(\bar p_i, i)  \quad \text{ for } s \in [-i, 0].$$

By the scaling invariance of \eqref{vol_est_cor3}, the corresponding estimate holds with $B_{\tilde g_i(0)}(\bar p_i, r)$ for any $0 < r \leq \sqrt{Q_i}$.
The rest of the proof goes exactly as the remaining steps in the proof of Proposition \ref{Int_BC}.

\subsection{Estimates for the extinction time}
We first need a scale invariant version of Petrunin's  estimate (Theorem \ref{Pet}).
\begin{lem}\lb{Pet_si}
Let $(M^n, g)$ be an open manifold with $K_g \geq 0$. Then for any $p\in M$ and $r > 0$, there exists a constant $C_n > 0$ such that
$$\int_{B_g(p, r)} \scal_g \, d\mu_g \leq C_n \,r^{n -2}.$$
\end{lem}

\bdem
For any $r > 0$, consider the rescaled metric $\tilde g = \frac1{r^2} g$. Since $K_{\tilde g} = \frac1{r^2} K_g \geq 0$, we are in position to apply Theorem \ref{Pet} to $(M, \tilde g)$ which gives
$$C_n \geq \int_{B_{\tilde g}(p, 1)}  \scal_{\tilde g}\, d\mu_{\tilde g} = \int_{B_g(p, r)} r^{2 - n} \scal_g \, d\mu_g,$$
where for the last equality we have used the identities $d\mu_{\tilde g} = r^{-n} \, d\mu_g$, $\scal_{\tilde g} = r^2 \scal_g$ and $B_{\tilde g}(p, 1) = B_g(p, r)$. 
\edem

\begin{lem}\label{vol collapse} Suppose $(M^n,g(t))_{t\in [0,T)}$ is a maximal solution of the Ricci flow with $K^\ce_{g(t)} \geq 0$. 
If $T<\infty$, then 
\[
 \limsup_{t\to T}\,\,\sup\Bigl\{\tfrac{\vol_{g(t)}(B_{g(t)}(p,r))}{r^{n-2}}\mid p\in M,r>0\Bigr\}=0.
\]
\end{lem}
\begin{proof}
We assume on the contrary that we can find $v_0>0$, $x_j \in M$, 
$t_j\to T$ and $r_j>0$ satisfying $\vol_{g(t_j)}(B_{g(t_j)}(x_j,r_j))\ge v_0r_j^{n-2}$. 
We fix some $(\bar x, \bar t,\bar r)=(x_{j_0},t_{j_0},r_{j_0})$ 
with $(T-\bar t)\le \tfrac{v_0}{2C_n}$, where 
$C_n$ is the constant in Lemma~\ref{Pet_si}.

Now we can use Petrunin's result 
 as in Lemma~\ref{Pet_si}
in order to estimate
\begin{eqnarray*}
\vol_{g(t)}(B_{g(t)}(\bar p,\bar r))&\ge&\vol_{g(t)}(B_{g(\bar t)}(\bar p,\bar r))\\ &\ge& (v_0- C_n(t-\bar t))\bar r^{n-2}\ge \tfrac{v_0}{2} \bar r^{n-2}\hspace*{1em}\mbox{ for $t\in [\bar t,\mc T)$}.
\end{eqnarray*}

This in turn allows us to prove, similarly to Proposition~\ref{Int_BC},
 that for each $D$ there is a $C_D$ 
with 
\[
|\Rm_{g(t)}|\le C_D \qquad \mbox{ on $B_{g(t)}(\bar p,D)$, \qquad $t\in \bigl[\bar t, \mc T\bigr)$.}
\]
As in the proof of Theorem~\ref{mainT_pos} 
we get also bounds on the derivatives of $\Rm_{g(t)}$. 
This in turn shows that $g(t)$ converges smoothly 
to a smooth limit metric $g(\mc T)$. 
By Lemma~\ref{lem: complete2} $g(\mc T)$ is complete 
and thus we can extend the Ricci flow by applying Theorem~\ref{mainT} 
to $(M,g(\mc T))$ -- a contradiction.
\end{proof}

\begin{proof}[\bf Proof of Corollary~\ref{thm: immortal}.]
Consider a maximal Ricci flow $(M,g(t))_{t\in [0,T)}$ with $K^\C\ge 0$
and suppose on the contrary that 
\[
T<\tfrac{1}{C_n}\sup\Bigr\{\tfrac{\vol_{g(0)}\(B_{g(0)}(p,r)\)}{r^{n-2}}\mid p\in M, r>0\Bigr\},
\]
where $C_n$ is the constant from Lemma~\ref{Pet_si}.

By assumption we can choose $ r>0$ 
and $p\in M$ with \[\vol_{g(0)}\bigl(B_{g(0)}(p,r)\bigr)> C_nT r^{n-2}.\] 
Using Petrunin's estimate (as restated in Lemma~\ref{Pet_si}) we deduce 
$$
\vol_{g(t)}\(B_{g(t)}(p,r)\)\ge \vol_{g(t)}\(B_{g(0)}(p,r)\)> C_n (T- \, t)\, r^{n-2}.
$$ 
Combining with Lemma~\ref{vol collapse} this gives a contradiction.

\end{proof}

\begin{proof}[\bf Proof of Corollary~\ref{thm: immortal3}]
Any open nonnegatively curved manifold 
with a two dimensional soul $\Sigma$ 
is locally isometric to $\Sigma\times \R$ 
and the Ricci flow exists exactly until 
$\tfrac{\mathrm{area}(\Sigma)}{4\pi\chi(\Sigma)}\in (0,\infty]$. 
If $\dim(\Sigma) = 1$, then 
the universal cover splits off a line and 
Corollary~\ref{thm: immortal} ensures the existence of an immortal solution. 
So it only remains to consider the case that the soul is a point.

If $T < \infty$, using Corollary~\ref{thm: immortal} we know
that 
\[
\limsup_{r\to \infty} \tfrac{\vol_{g}(B_{g}(p,r))}{r}= L < \infty.
\]
By Lemma \ref{Drees} there is a sequence 
$\eta_i\searrow 1$ such that for a base point $o\in M$ 
the following holds 
\[
B_{g}(o,i)\subset C_{i}\subset B_{g}(o,i\eta_i)\qquad \mbox{for all \quad$i\ge 1$,}
\] 
where $C_i$ is the sublevel set $b^{-1}((-\infty,i])$ 
 of the Busemann function at $o$ (see subsection~\ref{subsec: chgr}). 
In particular, $\tfrac{\vol_{g}(C_i)}{\vol_{g}(B_{g}(o,i))}\to 1$.

Clearly $\vol(C_i)=\vol(C_0)+\int_0^i\mathrm{area}(\partial C_t)\,dt.$
Moreover, by work of Sharafutdinov (see e.g.~\cite[Theorem 2.3]{Yim2})
there is an $1$-Lipschitz map $\partial C_b \fle \partial C_a$
 for $a\le b$. Accordingly, 
the area of $\partial C_i$ is monotonously increasing and 
thus 
\[
0 < \lim_{r\to \infty} 
\mathrm{area}(\partial C_r) = \lim_{r\to \infty} \tfrac{\vol_{g}(B_{g}(p,r))}{r}=L.
\]
This yields that $D := \lim_{r\to\infty}\diam(\partial C_r)< \infty$. 
In fact, suppose for a moment $D =\infty$. 
Choose $a>0$ so large that $\mathrm{area}(\partial C_a)\ge \tfrac{3}{4}L$. 
Since $\diam(\partial C_r)$ tends to infinity while the area converges to $L$,
we can find for each $\eps>0$ an $r$ and a circle of length $\le \eps$
in $\partial C_r$ which subdivides $\partial C_r$ into 
two regions of equal area. If we consider the image of this 
circle under the 1-Lipschitz map 
$ \partial C_r\rightarrow \partial C_a$ (for $r \geq a$), we get a closed curve
of length $\le \eps$ which subdivides $\partial C_a$ in two regions such 
that each of them has area at least $L/4$. 
Since $\eps$ was arbitrary this gives a contradiction.

As it is the boundary of a totally convex set, $\partial C_r$ 
is a nonnegatively curved Alexandrov space (cf.~Buyalo \cite{Buy}). Combining compactness and Sharafutdinov retraction,  $\partial C_r$ converges for $r\to \infty$ 
to a nonnegatively curved Alexandrov space $S$.
Moreover, for any sequence $p_i\in M$ converging to 
infinity we have 
$\lim_{GH,i\to \infty}(M,g,p_i)\to S\times \R$. 
Thus $M$ is asymptotically cylindrical. 

In particular, $M$ is volume non-collapsed, and from Corollary~\ref{cor: noncollapsed} we deduce 
that $(M,g(t))$ has bounded curvature $\le \tfrac{C}{t}$ for positive 
times. 
It is now easy to extract from the sequence 
$(M,g(t),p_i)$ a subsequence converging to $(N,g_\infty(t))$. 
Topologically the nonnegatively curved manifold 
$N$ is homeomorphic to $S\times \R$ -- a manifold
with two ends. Thus $(N,g_\infty(t))$ splits 
isometrically as $(\Sph^2,\bar g(t))\times \R$. 

From  Lemma~\ref{vol collapse} one can deduce that 
$\lim_{t\to T}\vol_{\bar g(t)}(\Sph^2)=0$. 
By Gau{\ss} Bonnet $\lim_{t\to 0}\vol_{\bar g(t)}(\Sph^2)=8\pi T=L$.

\end{proof}

\begin{rem}\label{rem: euclidean} Let $(M,g)$ be an open manifold with $K^\ce_g \geq 0$ and Euclidean volume growth.
By Corollary~\ref{cor: noncollapsed} 
the curvature of our Ricci flow $g(t)$ starting on $(M, g)$ is bounded for positive times. 
Following the work of Schulze and Simon \cite{SS}, with Hamilton's 
Harnack inequality replaced by \cite{BreHar},
one can show  
that there is a sequence of positive numbers $c_i\to 0$ 
 such that $\lim_{i\to\infty} (M,c_ig(t/c_i))=(M,g_\infty(t))$ 
is a  Ricci flow ($t\in (0,\infty)$) whose \lq initial  metric'
(Gromov Hausdorff limit of $(M,d_{g_\infty(t)})$ for $t\to 0$)
is the cone at infinity of $(M,g)$. 
Moreover, $(M,g_\infty(t))$ is an expanding
gradient Ricci 
soliton.
\end{rem}

\subsection{Long time behaviour of the Ricci flow.}
We will only consider solutions which satisfy the trace Harnack inequality. 
Notice that this is automatic if 
we consider a solution coming out of the proof of Theorem~\ref{mainT}. 

\begin{lem}\label{volume scale decay} Let $(M^n,g(t))$ be a non flat immortal solution 
of the Ricci flow with $K^\C\ge 0$ satisfying the trace Harnack 
inequality. If $(M,g(0))$ does not have Euclidean 
volume growth, then for $p_0\in M$ 
 \[
\limsup_{t\to \infty} \tfrac{\vol_{g(t)}(B_{g(t)}(p_0,\sqrt{t}))}{\sqrt{t}^n}=0.
 \]
\end{lem}
\begin{proof} Suppose on the contrary that we can find $t_k\to \infty$ 
and $\eps>0$ with $\vol_{g(t_k)}(B_{g(t_k)}(p_0,\sqrt{t_k}))\ge \eps \sqrt{t_k}^n$.
Analogous to Proposition~\ref{Int_BC}
one can show that there is some universal $\mc T > 0$ such that 
for the rescaled flow 
$\tg_k(t)= \tfrac{1}{t_k}g(t_k+ t \cdot t_k)_{t\in [-1,\infty)}$ 
we have that $\scal_{\tg_k(t)}\le \tfrac{C}{t}$ on 
$B_{\tg_k(t)}(p_0,1)$ for $t\in (0,\mc T]$ where $C$ is independent of $k$. 
Using this for $t=\mc T$ and combining with the Harnack inequality,
we find a universal constant $C_2$ 
with $\scal_{\tg_k(0)}\le C_2$ on $B_{\tg_k(0)}(p_0,1)$. 

Thus we obtain that $\scal_{g(t_k)}\le \tfrac{C_2}{t_k}$ 
on $B_{g(t_k)}(p_0,\sqrt{t_k})$. 
Combining with the Harnack inequality 
and using $B_{g(0)}\subset B_{g(t_k)}$ we deduce 
that $\scal_{g(t)}\le \tfrac{C_2}{t}$ on $M$. 
By Hamilton \cite[Editor's note 24]{Hamform} this implies that 
$B_{g(t)}(p_0,\sqrt{t})\subset B_{g(0)}(p_0, C_3\sqrt{t})$
for a constant $C_3=C_3(C_2,n)$. 
Hence
\[
\eps \sqrt{t_k}^n\le \vol_{g(t_k)}\Bigl( B_{g(0)}(p_0,C_3\sqrt{t_k})\Bigr)\le 
\vol_{g(0)}\Bigl( B_{g(0)}(p_0,C_3\sqrt{t_k})\Bigr),
\]
which means that $g(0)$ has Euclidean volume growth -- a contradiction.
\end{proof}

\begin{thm}\label{thm: steady} Let  $(M^n,g(t))$ be a non flat immortal Ricci flow with $K^\C\ge 0$ satisfying the trace Harnack 
inequality. If $(M,g(0))$  does not have Euclidean 
volume growth, then for $p_0\in M$ there is a sequence of times 
$t_k\to \infty$ and a rescaling sequence $Q_k$ 
such that for $\tg_k(t)= Q_kg(t_k+\tfrac{t}{Q_k})$ 
the following holds. The rescaled flow 
$(M,\tg_k(t),p_0)$ converges in the Cheeger-Gromov sense 
to a steady soliton
$(M_\infty,\tg_\infty(t))$ which is not isometric 
to $\R^n$.
\end{thm}

\begin{proof} For $t\in [0,\infty)$ we define $Q(t)>0$ 
as the minimal number for which
\[
\vol_{g(t)}\Bigl( B_{g(t)}\bigl(p_0,\tfrac{1}{\sqrt{Q(t)}}\bigr)\Bigr)=
\tfrac{1}{2}\omega_n\tfrac{1}{\sqrt{Q(t)}^n}
\]

We can choose $\eps_k\to 0$ and $t_k\to \infty$ 
with $\tfrac{\partial}{\partial t}_{|t=t_k} \scal_{g(t)}(p_0)\le \eps_k\scal_{g(t_k)}(p_0)^2$. 
In fact otherwise it would be easy to deduce that a finite 
time singularity occurs. 

By Lemma~\ref{volume scale decay}
the rescaled Ricci flow
$\tg_k(t)= Q_k\, g(t_k+\tfrac{t}{Q_k})$ with $Q_k=Q(t_k)$
is defined on an interval $[-T_k,\infty)$ with $T_k\to \infty$.
Moreover, $\vol_{\tg_k(0)}\(B_{\tg_k(0)}(p_0,1)\)=\tfrac{\omega_n}{2}$ 
and $\tfrac{\partial}{\partial t}_{|t=0} \scal_{\tg_k(t)}(p_0)\le
 \eps_k\scal_{\tg_k(0)}(p_0)^2$.

Arguing as in the proof of
 Lemma~\ref{volume scale decay} one can show 
that there is some $\mc T>0$ such that for each $r$ there is
 a constant $C_r$ for which
$\scal_{\tg_k(\mc T)}\le C_r$ on $B_{\tg_k(\mc T)}(p_0,r)$.
Using the Harnack inequality  
after possibly increasing $C_r$ we may assume 
that $\scal_{\tg_k(t)}\le C_r$ on $B_{\tg_k(\mc T)}(p_0,r)$ 
for all $t\in [-T_k/2,\mc T]$. 
Shi's estimate also give bounds on the derivative of the 
curvature tensor on $B_{\tg_k(\mc T)}(p_0,r)\times [-T_k/4,\mc T/2]$. 

After passing to a subsequence 
we may assume $(M,\tg_k(\mc T/2),p_0)$ 
converges in the Cheeger-Gromov sense 
to $(M_\infty,\tg_\infty(\mc T/2),p_\infty)$. 
By the Arcel$\grave{a}$-Ascoli theorem we also 
may assume that under the same set of local diffeomorphisms 
the pull backs of
$\tg_k(t)$ converge to 
$\tg_\infty(t)$, ${t\in (-\infty,\mc T/2]}$. 
Clearly $\tg_\infty(t)$ is a solution of the Ricci 
flow with $K^\C_{\tg_\infty(t)}\ge 0$. The completeness 
of $\tg_\infty(t)$ follows from the completeness of $ \tg_\infty(\mc T/2)$, 
for $t<\mc T/2$.
Moreover, we have that $\tfrac{\partial\scal_{\tg_\infty(0)}(p_0)}{\partial t}=0$. 
If $\tg_\infty(0)$ is not flat, we can pass to the universal cover of $M_\infty$ 
and after spitting off an Euclidean factor we may assume that the Ricci 
curvature is positive.
Recall that for ancient solutions with $K^\C_{\tg_\infty(t)}\ge 0$
 and positive Ricci curvature the Harnack inequality implies that 
\[
0\le \tfrac{\partial\scal_{\tg_\infty(t)}}{\partial t}-\tfrac{1}{2}\Ric^{-1}_{\tg_\infty(t)}(\nabla \scal_{\tg_\infty(t)},\nabla \scal_{\tg_\infty(t)}),
\]
where $\Ric^{-1}_{\tg_\infty(t)}$ is the (positive definite) symmetric
 $(2,0)$-tensor defined by the equation
$\Ric^{-1}_{\tg_\infty(t)}\bigl(v,\Ric^{\tg_\infty(t)}w\bigr)=\tg_\infty(t)(v,w)$.
Since equality holds for one point in space-time one can deduce 
from a strong maximum principle, 
that $\tilde g_\infty(t)$ is a steady Ricci soliton.
In fact, this only requires minor modification in the proof 
of a result by Brendle \cite[Proposition 14]{BreHar}. We 
leave the details to the reader.
Finally $\vol_{\tg_\infty(0)}\bigl(B_{\tg_\infty(0)}(p_0,1)\bigr)=\tfrac{1}{2}\omega_n$ and thus
the limit is not the Euclidean space.
\end{proof}

\subsection{Further Consequences} 
\begin{cor}\label{cor: deform} Let $(M,g)$ be an open manifold
with $K^{\C}_g\ge 0$ and $M\cong \R^n$, then 
there is a sequence $g^i$ of complete metrics on $M$
with $K^{\C}_{g^i}> 0$ converging to $g$ in the 
$C^\infty$ topology.
\end{cor}
\begin{proof} Consider the de Rham decomposition
$M=\R^k\times (M_1,g_1)\times \ldots\times (M_l,g_l)$ of $M$. 
Let $g_j(t)$ be a Ricci flow from Theorem \ref{mainT}
 on $M_j$ with $g_j(0) = g_j$, and let $g(t)$ be the corresponding product metric on 
$M$.
 We  know that $K^{\C}_{g_j(t)}\ge 0$
and clearly $(M_j,g_j(t))$ is irreducible for 
small $t\ge 0$. 
Since $M_j$ is diffeomorphic to a Euclidean space, 
$(M_j,g_j(t))$ cannot be Einstein and we can deduce 
from Berger's holonomy classification theorem \cite{Berh}
that the holonomy group is either $\SO(n_j)$ 
or $\gU(n_j/2)$, where $n_j = \dim M_j$.

The strong maximum principle implies 
that either $ K^{\C}_{g_j(t)}> 0$ or $(M_j,g_j(t))$ 
is K\"ahler. 
Even in the K\"ahler case it follows
 that the (real) sectional curvature is positive,
$K_{g_j(t)}>0$: If there were a real plane 
with $K_{g_j(t)}(\sigma)=0$, then by the strong maximum
principle we could deduce that either $K_{g_j(t)}(v,Jv)=0$ 
for all $v\in TM$ or $ K_{g_j(t)}(v,w)=0$ for all 
$v, w\in T_pM$ with ${\rm span}_{\R} \{v, J v\} \perp {\rm span}_{\R} \{w, J w\}$. But both conditions 
imply in $K \geq 0$ that the manifold is flat. 

Since $K_{g_j(t)}>0$, Theorem~\ref{GW1t} of Greene and Wu gives a strictly convex smooth 
proper nonnegative function $b_j(t)$ on $(M_j,g_j(t))$. 
Clearly we can also find such a function on $\R^k$. 
By just adding these functions we deduce 
that there is a proper function $b(t)\colon M\rightarrow [0,\infty)$ 
which is strictly convex with respect to the product 
metric $g(t)$. 
We now choose a sequence $t_i\to 0$ and $\eps_i\to 0$ 
and define $g^i$  as the metric on $M$ which is obtained by pulling back the 
metric on the graph of 
$\eps_i b(t_i)$  viewed as hypersurface in $(M,g(t_i))\times \R$. 
Clearly $K^\C_{g^i}>0$ and if $\eps_i$ tends to 
zero sufficiently fast, then $g^i$ converges to $g$ 
in the $C^\infty$ topology.
\end{proof}

\begin{rem} Although 
a priori we could prove this only for large $\ell$, 
it is true that for each convex set $C_{\ell}=b^{-1}((-\infty,\ell])$ 
one can find a Ricci flow on a compact manifold 
with $K^\C\ge0$ such that 
$(M,g(t))$ converges to the
 double $D(C_\ell)$ of $C_{\ell}$ for $t\to 0$. 
In fact, by using Corollary~\ref{cor: deform} one 
can find a sequence of strictly convex sets $C_{\ell,k}$
 in manifolds with $K^\C>0$ 
which converge in the Gromov-Hausdorff topology to $C_{\ell}$. 
For strictly convex sets it is not hard to see 
that one can smooth the double $D(C_{\ell,k})$ without losing $K^\ce \geq 0$ and thus the result follows. 
\end{rem}

\section{ An immortal nonnegatively curved solution of the Ricci flow with 
unbounded curvature}\label{sec: immortal}

\subsection{Double cigars}

Recall that Hamilton's cigar is the complete Riemannian surface $(C, g_0) :=\(\R^2, \frac{dx^2 + dy^2}{1 + x^2 + y^2}\)$, which is rotationally symmetric, positively curved and asymptotic at infinity to a cylinder of radius $1$. The Ricci flow starting at $(C, g_0)$ is a gradient steady Ricci soliton (i.e.~an eternal self-similar solution). 
\begin{defin} Let $(M,\bar g)$ and $(N,g)$ be two complete 
$n$-dimensional Riemannian manifolds, $p\in M$ and $q\in N$. 
We say $(M,\bar g,p)$ is $\eps$-close to $(N,g,q)$ if 
\begin{enumerate}
\item[$\circ$] there is a subset $U\subset N$ with 
$ B_{\frac{1}{\eps}-\eps}(q)\subset U\subset B_{\frac{1}{\eps}+\eps}(q)$ and
\item[$\circ$]  a diffeomorphism $f\colon B_{\frac1{\eps}}(p)\rightarrow U$ such that $ \|\bar g-f^{*}g\|_{C^k}\le \eps$ for all $k\le 1/\eps$.
\end{enumerate}
\end{defin}

We denote $x_0\in C$ 
the tip of Hamilton's cigar, i.e.~the unique fixed point of the isometry group $\mathrm{Iso}(C)$, where the maximal 
curvature of $C$ is attained.
We will also consider the rescaled manifolds $(C,\lambda^2 g_0)$.
\begin{defin} A nonnegatively curved metric $g$ on $\Sph^2$ 
is called an $(\eps,\lambda)$-double cigar if the following holds
\begin{enumerate}
 \item[$\circ$] $g$ is invariant under  $O(2)\times \Z_2\subset O(3)$, and
\item[$\circ$] if $\bar p$ is one of the two fixed points of the identity component  
of  $O(2)\times \Z_2$, then $(\Sph^2,g, \bar p)$ is $\eps$-close 
to $(C,\lambda^2g_0,x_0)$.
\end{enumerate}
\end{defin}

An important feature of the definition is that except for nonnegative curvature, 
 we do not make  any assumptions on the middle region of the double cigar. 
In the applications we will have $\diam(\mathbb S^2,g)\gg\tfrac{1}{\eps}$. 

We have two easy consequences of compactness results.

\begin{lem}\label{lem: double cigar} For any $\lambda$ and $\eps>0$ there exists some $\delta>0$ 
such that: If $(\Sph^2,g)$ is any 
$(\delta,\lambda)$-double cigar and $(\Sph^2,g(t))$ is a Ricci flow with $g(0)=g$, then $ (\Sph^2,g(t))$
 is an $(\eps,\lambda)$-double cigar
for all $t\in [0,1/\eps]$.
 
\end{lem}

\begin{lem}\label{lem: stable double cigar} Let $\bar g$ be a nonnegatively curved metric on $\Sph^2$, $(\Sph^2, \bar g(t))_{t\in [0,T]}$ the Ricci flow with $\bar g(0)= \bar g$, and $\bar p\in \Sph^2$.
For a given $\eps > 0$ there exists a positive integer $\delta=\delta(\eps, \bar g)$ 
such that the following holds. 

Let $(M^3,g)$ be any open nonnegatively curved 3-manifold 
and $p\in M$ so that $(M,g,p)$ is $\delta$-close to $\bigl((\Sph^2,\bar g)\times \R,(\bar p,0)\bigr)$. 
If $(M,g(t))$ is an immortal nonnegatively curved  Ricci 
flow with $g(0)=g$, then 
$(M,g(t),p)$ is $\eps$-close to $ \bigr((\Sph^2,\bar g(t))\times \R,(\bar p,0)\bigr)$ 
for all $t\in [0,1/\eps]\cap [0,T/2]$.
\end{lem}

\begin{proof}[Proof of Lemma~\ref{lem: double cigar}]
Suppose on the contrary that for some positive $\eps$ and $\lambda$
we can find a sequence of $(\tfrac{1}{i},\lambda)$-double 
cigars $(\Sph^{2},g_i)$ and times $t_i\in [0,\tfrac{1}{\eps}]$ such that 
$(\Sph^2,g_i(t_i))$ is not an $(\eps,\lambda)$-double cigar. 
Here $(\Sph^2, g_i(t))_{t\in [0, T_i)}$ is the maximal solution of the Ricci flow with $g_i(0)= g_i$ .  
Let $\bar p$ denote a fixed point of the identity component
of the $O(2)\times \Z_2$-action.
By assumption we know that $(\Sph^2,g_i(t_i),\bar p)$ 
is not $\eps$-close to $(C,\lambda^2g_0,x_0)$.

It is easy to see that the volume of 
 any unit ball in $(\Sph^{2},g_i)$ is bounded below by a universal constant
independent of $i$. Thus we have universal curvature 
and injectivity radius bounds for all positive times. 
Moreover, by Gau{\ss} Bonnet $T_i\to \infty$. 
Using furthermore that we have control of 
the curvature and its derivatives on larger and larger balls 
around $\bar p$, one can deduce that 
the Ricci flow subconverges to a (rotationally symmetric) limit immortal
solution on the cigar  $(C,g_\infty(t),x_0)$ with 
bounded curvature and whose initial metric is $\lambda^2 g_0$. 
Because of the uniqueness of the Ricci flow (see \cite{ChenZhu}) it follows that 
$(C,g_\infty(t),x_0)$ is isometric to $(C,\lambda^2g_0,x_0)$ for all 
$t$. On the other hand, if $t_\infty\in [0,1/\eps]$ 
is a limit of a convergent subsequence of $t_i$, then 
$(C,g_\infty(t_\infty),x_0)$ is not $\eps/2$-close to $(C,\lambda^2g_0,x_0)$
-- a contradiction.
\end{proof}

\begin{proof}[Proof of Lemma~\ref{lem: stable double cigar}]

Suppose on the contrary that we can find a sequence 
 $(M_i,g_i)$ of open $3$-manifolds with $K_{g_i} \geq 0$ 
and $p_i\in M_i$ such that $(M_i,g_i,p_i)$ is $\tfrac{1}{i}$-close to $\bigl((\Sph^2,\bar g)\times \R,(\bar p,0)\bigr)$
and a complete immortal Ricci flow  $g_i(t)$ 
 with $g_i(0)=g_i$ and $K_{g_i(t)} \geq 0$, such that 
$(M_i,g_i(t_i),p_i)$ is not 
$\eps$-close to $ \bigr((\Sph^2,\bar g(t))\times \R,(\bar p,0)\bigr)$
for some $t_i \in [0, T/2]$.

Arguing similar to the proof of the previous lemma, we can use Hamilton's compactness theorem  to deduce 
that $(M_i,g_i(t),p_i)$ converges to a limit 
nonnegatively curved solution on the manifold 
$\Sph^2\times \R$.   Clearly the solution 
is just given by the product solution 
on $\Sph^2\times \R$ and because of uniqueness of the 
Ricci flow on $\Sph^2$ we deduce that it is exactly given by 
$\bigr((\Sph^2,\bar g(t))\times \R,(\bar p,0)\bigr)$ -- 
again this yields a contradiction.
\end{proof}

\subsection{Convex hulls of convex sets}\label{subsec: convex hulls}
Let $C_0$ and $C_1$ be two closed convex sets of $\R^n$. 
Then 
\[C_\lambda=\{ (1-\lambda) x+ \lambda y\mid x\in C_0,y\in C_1\}\] 
is convex as well. 
If $\partial C_0$ and $\partial C_1$ are smooth compact
hypersurfaces of positive sectional curvature, 
then $\partial C_\lambda$ is smooth as well:
In fact, let $N_0$ and $N_1$ denote the unit outer normal fields
of $C_0$ and $C_1$. 
By assumption
 $N_i\colon \partial C_i\rightarrow \Sph^{n-1}$
is a diffeomorphism, $i=0,1$. 
For $z=(1-\lambda) x+ \lambda y\in C_\lambda$ and $\lambda\in (0,1)$ 
the tangent cone $T_zC_\lambda$ contains $T_xC_0$
as well as $ T_yC_1$. This in turn implies
\[
 \partial C_\lambda=\bigl\{ (1-\lambda) x+ \lambda N_1^{-1}(N_0(x))\mid x\in \partial C_0\bigr\}
\]
and thus $\partial C_\lambda$ is smooth.
Furthermore, it is easy to see that $\partial C_\lambda$ is positively curved as
well. 

Consider now the convex sets $C_0\times \{h_0\}$ and $C_1\times \{h_1\}$ 
in $\R^{n+1}$. The convex hull 
$C$ of these two sets is given by 
\[
 C= \bigl\{ (y,(1-\lambda) h_0+\lambda h_1)\mid y\in C_\lambda,\lambda\in[0,1] \bigr\}.
\]
In particular, we see that the boundary $\partial C\cap \(\R^n\times (h_0,h_1)\)$ 
is a smooth manifold.

\subsection{Proof of Theorem~\ref{thm: exa} a)}\label{subsec: constr}
By Lemma~\ref{lem: double cigar} 
we can find a sequence $\eps_i\to 0$ such that 
any $(\eps_i,\tfrac{1}{i})$-double cigar $(\Sph^2,g)$ satisfies the following:
The solution of the Ricci flow $g(t)$ with $g(0)=g$ exists on
$[0,i]$ and $(\Sph^2,g(t))$ is a $(2^{-i},\tfrac{1}{i})$-double cigar for 
all $t\in [0,i]$. 
The sequence $\eps_i$ is hereafter fixed.

We now define inductively a sequence of
$(\eps_i,\tfrac{1}{i})$-double cigars $S_i$ so that
\begin{enumerate}
 \item[(1)] $S_i$ has positive curvature and 
embeds as a convex hypersurface $S_i \subset\R^3$ 
in such a way that it is invariant under the linear action 
of $\Z_2\times \Or(2)\subset \Or(3)$.
\item[(2)] The convex domain bounded by $S_{i-1}$ is contained in the interior of the convex 
domain bounded by $S_{i}$.
\end{enumerate}

It is fairly obvious that one can find $(\eps_i,\tfrac{1}{i})$-double cigars
satisfying (1). 
In order to accomplish also (2), we 
choose $r_0$ such that $S_{i-1}\subset B_{r_0}(0)$. 
We can find an $(\eps_{i},\tfrac{1}{i})$-double cigar $(\Sph^2,g)$ 
and a fixed point $\bar p\in \Sph^2$ of the identity component of
$ \Z_2\times \Or(2)$ 
such that 
\begin{itemize}
\item[$\circ$] $B_{1/\eps_{i}}(\bar p)$ is isometric to $B_{1/\eps_i}(x_0)\subset (C,\tfrac{1}{i^2}g_0)$, 
\item[$\circ$] for some $R \gg \frac{1}{
\eps_i}$ the set $ B_{R}(\bar p)\setminus B_{2/\eps_i}(\bar p)$ 
is isometric to a subset of a cone  
\item[$\circ$] and $B_{4R}(\bar p)\setminus B_{2R}(\bar p)$ is isometric 
to $\Sph^1\times [0,2R)$ for a circle of length $4\pi r_0$. 
\end{itemize}
We can now construct an embedding of this cigar into 
$\R^3$ such that the surface $S_i$ bounds a convex domain which contains 
$B_{2r_0}(0)$. By slightly changing the embedding one can ensure that $S_i$ has 
 positive sectional curvature.
The sequence $S_i$ of embedded double cigars is now fixed.

We put $S_0=\{0\}$. 
We now define inductively  a sequence of positive numbers $r_i\to \infty$, and heights  $h_i:=\sum_{j=0}^i r_j$. Denote
 $C_j\subset \R^4$  the convex hull of 
$S_{j-1}\times \{h_{j-1}\}\subset \R^4$ and  $S_{j}\times \{h_j\}\subset \R^4$. By choosing $r_{i}$ large enough we can 
arrange for the following: 
\begin{itemize}
\item[$\circ$] The union $C_{i-1}\cup C_{i}$ is convex as well.
In fact,  $C_{i}$ converges to the cylinder  
bounded by $S_{i-1}\times [h_{i-1},\infty)$ 
for $r_{i}\to \infty$.
Hence for all large $r_{i}$ and all $p\in S_{i-1} \times \{h_{i -1}\}$ 
the union of the tangent cones $T_pC_{i-1}$ and $T_pC_{i}$ 
is properly contained in a half space.

\item[$\circ$] The hypersurface $H_i:= \bigl(\R^3\times [h_i- 1 -\sqrt{r_i},h_i-1]\bigr)\cap \partial C_i$
is arbitrarily close to a product $S_i\times [h_i- 1 -\sqrt{r_i},h_i-1]$ 
in the $C^\infty$ topology.

\item[$\circ$] 
For any open $3$-manifold $(M^3, \tilde g)$ with $K_{\tilde g} \geq 0$ containing an open subset 
$U$ isometric to $H_i$, for big $r_i$ Lemma~\ref{lem: stable double cigar} ensures that, if $(M, \tilde g(t))$ is an  immortal Ricci 
flow with $\tilde g(0)= \tilde g$ and $K_{\tilde g(t)} \geq 0$, then for some  $p\in U$ we have that 
$(M, \tilde g(t),p)$ is $\tfrac{1}{(1+\diam(S_i))^i}$-close to $( (\Sph^2,g(t)) \times \R,(\bar p, 0))$, 
where $(\Sph^2,g(t))$ is a $(2^{-i},\tfrac{1}{i})$-double cigar
for all $t\in [0,i]$. 
\end{itemize}

By construction, $C=\bigcup_{i\ge 1} C_i$ is a convex set whose boundary $\partial C$ is not smooth but the singularities 
only occur for points in $\R^3\times \{h_i\}\cap \partial C$,
see subsection~\ref{subsec: convex hulls}. 
We can now smooth $C$ as follows:
Notice that $\partial C\subset \R^4 $ can be defined as the graph of a convex
 function $f$ on $\R^3$. By construction
$T_p\partial C$ is not a half space for all $p\in S_i\times \{h_i\}$ 
and thus the gradient of $f$ jumps at the level set $S_i$. 

 We choose a smooth convex function 
$g\colon [0,\infty)\rightarrow [0,\infty)$ 
with $g\equiv 0$ on $[1,\infty)$ and $g''>0$ on $[0,1)$. 
We also define
$\varphi\colon [0,\infty) \rightarrow \R$ by 
$\varphi(t)=t+\delta\cdot g(|t-h_i|)$ for some  $\delta>0$ 
which is to be determined next. 
Notice that $\varphi$ is 
is $C^\infty$ on $[0,h_i]$ and on $[h_i,\infty)$. 
Moreover $\varphi(t)=t$ for $|t-h_i|>1$. 
On the intervals $[h_i-1,h_i]$ and $[h_i,h_i+1]$ 
the function $\varphi$ is convex. However, at $h_i$
left and right derivative of $\varphi$
differ by $2\delta g'(0)$.
Similarly there is small neighborhood $U$ of $S_i$ 
such that $f$ is $C^\infty$ on $U\setminus S_i$. 
Along the level set $S_i$ there is an outer and an inner gradient 
of $f$ and the norm of the outer gradient 
is strictly larger than the norm of the inner gradient. 
Thus
we can choose $\delta$ so small that 
 $\varphi \circ f$ is still 
a convex function. 
In addition, we know that 
the Hessian of $\varphi \circ f$ is bounded below by a small 
positive constant in a neighborhood of $S_i$.
Thus we can mollify $\varphi \circ f$ in a neighborhood of $S_i$ 
and patch things together using a cut off function.

By doing this procedure iteratively for all $i$ 
we obtain a smooth convex hypersurface $H$.
By construction the volume growth of $H$ is larger than 
linear and by Corollary~\ref{thm: immortal} we have an immortal solution $g(t)$ 
of the Ricci flow on $H$ starting with the initial metric.

By construction for any $i$ we can find a point 
$p\in H$ such that 
$(H,g(t),p)$ is $\tfrac{1}{i}$-close to $((\Sph^2,g(t)) \times \R,(\bar p, 0))$ 
where $(\Sph^2,g(t))$ is a $(2^{-i},\tfrac{1}{i})$-double cigar
for all $t\in [0,i]$. In particular, we know that 
\begin{eqnarray*}
 \sup\bigl\{K_{g(t)}(\sigma)\mid \sigma\subset TH \bigr\}&\,=\,\,\infty &\mbox{ and} \\
\inf\bigl\{ \vle_{g(t)}(B_{g(t)}(p, 1))\mid p \in H\bigr\}&\!=\,\,0 &\mbox{ for all $t$.}
\end{eqnarray*}

\begin{rem}
(a) A volume collapsed nonnegatively curved 
$3$-manifold was constructed by Croke and Karcher \cite{CK}.
Although the details are somewhat different, their example is realized as a convex hypersurface of $\R^4$ as well.

(b) 
At an informal discussion at UCSD the second named author 
was asked by Richard Hamilton, whether 
a nonnegatively curved three dimensional ancient solution
with unbounded curvature could exist. 
During this discussion Hamilton described 
possible features of a counterexample. 
 The construction in this section
 is in part inspired by what 
Hamilton had in mind.
Since the construction only gives an immortal
solution, Hamilton's question remains open nevertheless.

(c) As said in the introduction,
 a nonnegatively curved 
surface evolves immediately to bounded curvature under the Ricci flow. Giesen and Topping \cite{GieTo2} gave immortal $2$-dimensional Ricci flows
with unbounded curvature throughout time.
\end{rem}

\subsection{Proof of Theorem~\ref{thm: exa} b).}

\begin{lem}\label{lem: bounded} Let $(M,g)$ be an open manifold with $K^\C\ge 0$ 
and bounded curvature. 
Then there is an $\eps>0$ and $C$ such that, for any complete Ricci flow $g(t)$ with $g(0)=g$ 
and $K^\C_{g(t)}\ge 0$, we have $\scal_{g(t)} \leq C$   
on the interval 
$[0,\eps]$.
\end{lem}
\begin{proof} Recall that the injectivity radius of an open 
nonnegatively curved manifold with bounded curvature is positive. 
By Corollary~\ref{cor: noncollapsed} 
for any solution $g(t)$ we know
$\scal_{g(t)} \le \tfrac{C_1}{t}$ on some interval $(0,\eps]$. 
We can now use Theorem~\ref{SURF} to see that
$\scal_{g(t)} \le C$ for all $t\in [0,\eps]$ 
with some universal $C=C((M,g))$. 
\end{proof}

\begin{lem}\label{lem: 4 exa} There is an open $4$-manifold $(M,g)$
with nonnegative curvature operator and a constant $v_0>0$ such that 
the following holds
\begin{enumerate}
 \item[$\circ$] $\vol(B_g(p,1))\ge v_0$ for all $p\in M$.
 \item[$\circ$] There is a sequence of points $p_k\in M$ 
such that $(M,g,p_k)$ converges in the Cheeger-Gromov sense 
to the Riemannian product $\Sph^2\times \R^2$ 
(where $\Sph^2$ has constant curvature $1$ and $\R^2$ is flat).
\item[$\circ$] There is a sequence of points $q_k\in M$ 
such that $(M,g,q_k)$ converges in the Cheeger-Gromov sense 
to $\R^4$ endowed with the flat metric.
\end{enumerate}
\end{lem}

\begin{proof} The construction of $(M^4,g)$ is very similar 
to the one in subsection~\ref{subsec: constr}.
There is a sequence of embedded $(\tfrac{1}{i},1)$-double 
cigars $S_i\subset \R^3$ so that 
\begin{enumerate}
 \item[$\circ$] $S_i$ is invariant under an $\Z_2\times \Or(2)\subset \Or(3)$-action
fixing the origin. 
\item[$\circ$] The interior of the convex domain bounded by $S_i$ contains
$S_{i-1}$.
\item[$\circ$]  $S_i$ 
contains a subset which is isometric to $\Sph^1\times [-R_{i-1},R_{i-1}]$ 
where $\Sph^1$ is  a circle of radius $2R_{i-1}=2\diam(S_{i-1})\to \infty$.
\end{enumerate}

Analogous to subsection~\ref{subsec: constr}
one can then construct 
a smooth convex hypersurface $(M^3,g)\subset \R^4$ 
with $\Or(2)\times \Z_2$ symmetry satisfying: 
There is $p_i\in M$ such that $(M^3,p_i)$ is 
$e^{-R_i}$-close to $S_i\times \R$.
Moreover, it is clear from the construction  that  
$(M^3,g)$ is uniformly volume non collapsed.
We now define $M^4$ as the unique convex hypersurface 
in $\R^5$ whose intersection with $\R^4$ is given by 
$(M^3,g)$ and which has a $\Or(3)\times \Z_2$ symmetry.
It is straightforward to check that $M^4$ with the induced metric  
has the claimed properties.
\end{proof}

\begin{proof}[Proof of Theorem~\ref{thm: exa} b)]
Let $(M^4,g)$ be as in Lemma~\ref{lem: 4 exa}. 
We consider a solution $g(t)$ of the Ricci flow coming out of the proof 
of Theorem~\ref{mainT}. 
Using $(M^4,g, q_i)$ converges to $\R^4$ 
in the Cheeger-Gromov sense, 
it follows that the Ricci flow on the compact
approximations $(M_i,g_i(t))$ converging to $(M,g(t))$ 
exists until $T_i\to \infty$. 
By the proof of Theorem~\ref{mainT} we can assume that 
$g(t)$ is immortal. Moreover, it is clear from the proof that $g_i(t)$ and hence $g(t)$ 
have nonnegative curvature operator.
In particular we have an immortal complete 
solution $g(t)$ with $g(0)=g$ and $g(t)$ satisfies the trace Harnack 
inequality. 

By Corollary~\ref{cor: noncollapsed} it follows that $\scal_{g(t)}\le \tfrac{C}{t}$
for $t\in (0,\eps]$. 
We claim that $g(1)$
has unbounded curvature.
Suppose on the contrary that  $\scal_{g(1)} \leq C$. 
The trace Harnack inequality implies that  $\scal_{g(t)} \leq \tfrac{C}{t}$ for $t\in (0,1]$. 

We now consider the sequence $(M,g,p_i)$ 
converging to $(\Sph^2\times \R^2,p_\infty)$ in the Cheeger-Gromov sense. 
Applying Theorem~\ref{SURF} it is easy that there is a universal 
constant $C_2$ such that for any $r$ we can find $i_0$ 
such that $\scal_{g(t)}\le C_2$ on $B_{g(0)}(p_i,r)$ for all
$t\in [0,1]$ and $i\ge i_0$. 
By Hamilton's compactness theorem $ (M,g(t),p_i)$ subconverges to a solution $g_\infty(t)$ of the 
Ricci flow on $ (\Sph^2\times \R^2,p_\infty)$  with bounded curvature
such that $g_\infty(0)$ is given by the product metric 
($\Sph^2$ with constant curvature $1$), $t\in [0,1]$. 
On the other hand, for the
the unique solution (with bounded curvature) the curvature blows up 
at time $1/2$ -- a contradiction. 

In summary, we can say
$\scal_{g(t)}$ is bounded for $t\in (0,\eps]$ and that $\scal_{g(1)}$ 
is unbounded. By Lemma~\ref{lem: bounded} 
there must be a minimal time $t_0\in (\eps,1]$ such that $g(t_0)$ 
has unbounded curvature. From the trace Harnack inequality it follows 
that $g(t)$ has unbounded curvature for all $t\ge t_0$. 

Thus $M^4$ endowed with the
rescaled Ricci flow $\tg(t):=\tfrac{1}{t_0-\eps/2}g(\eps/2+t(t_0-\eps/2))$ 
satisfies the conclusion of Theorem~\ref{thm: exa} b).
\end{proof}

\appendix

\section{Open manifolds of nonnegative curvature} \lb{AppA}

Recall that a set $C$ is called totally convex if for every geodesic segment $\Gamma$ joining two points in $C$, we have $\Gamma \subset C$, and that the normal bundle of a submanifold $S \subset M$ is $\nu(S) = \bigcup_p \{v \in T_p M\, |\, v \perp T_p S\}$. 
We start with the Soul Theorem.
\bt[Cheeger-Gromoll-Meyer, cf.~\cite{ChGr, GroMe}] \lb{sl_thm} Let $(M^n, g)$ be an open ma\-nifold with $K_g \geq 0$. Then there is a closed, totally geodesic submanifold $\Sigma \subset M$ which is totally convex and with $0\leq {\rm dim} \Sigma < n$. $\Sigma$ is called a soul of $M$, and $M$ is diffeomorphic to $\nu(\Sigma)$. If $K_g > 0$, then a soul of $M$ is a point, and so $M$ is diffeomorphic to $\re^n$.
\et

Here is a further property about the soul.
\bt[Strake, cf.~\cite{Strake}] \lb{Strake-thm} 
Let $(M^n, g)$ be an open manifold with $K_g \geq 0$ and $\Sigma^k$ be the soul of $M$. If the holonomy group of $\nu(\Sigma)$ is trivial then $M$ is isometric to $\Sigma \times \re^{n - k}$, where $\re^{n - k}$ carries a complete metric of $K \geq 0$. 
\et
Fix $p \in \Sigma$ and let $d_p = d_g(\ccdot, p)$, where $d_g$ is the Riemannian distance. It is known (see e.g.~\cite{Drees}) that $b$ (see subsection \ref{subsec: chgr}) and $d_p$ are asymptotically equal:
\begin{lem} \lb{Drees}
 There exists a function $\theta(s)$ with $\lim_{s\to \infty}\theta(s) = 0$  such that
$$(1 - \theta \circ d_p) d_p \leq b \leq d_p,$$
and for all $x, y \in M$ it holds $|b(x) - b(y)| \leq d_g(x, y)$.
\end{lem}

It is useful to recall that $b$ is indeed the distance from an appropriate set:
\begin{lem}[Wu, cf.~\cite{Wu_el}] \lb{wu_le}
Let $a \in \re$ and let $C_a = \{ x \in M : b (x) \leq a\}$. Then  $b|_{int(C_a)} = a - d(\cdot, \partial C_a)$.
\end{lem}

\section{Convex sets in Riemannian manifolds} \lb{AppB}

Let $C$ be a compact totally convex set ({\sc tcs}) in a manifold $M$. We define the tangent cone at $p \in \partial C$ as
$$T_p C  =  {\rm Clos}\{v \in T_p M\, : \, \exp_p (t v/|v|) \in C \text{ for some } t > 0\}.$$
By convexity of $C$, this is a convex cone in $T_p M$. The normal space is defined as
\begin{align*}
N_p C  = \{v\in {\rm span}(T_p C) \, : \, \<v, w\> \le 0 \ \text{for all } w \in T_p C \setminus \{0\}\}.
\end{align*}

Here is a useful characterization of the normal space.
\bp[Yim, cf.~\cite{Yim2}] \lb{Yim}
Let $\{C_a\}$ be a family of {\sc tcs}. Consider $a > b$ with $a - b < \delta$, where $\delta > 0$ is chosen so that the projection $C_a \fle C_b$ is well-defined (i.e. for all $q\in C_a$ there is a unique $q^\ast \in C_b$ with $d(q, q^\ast) = d(q, C_b)$). For each $p \in \partial C_b$, 
$N_p C_b$ is the convex hull of the set of vectors $v \in \text{\rm span}(T_p C)$ such that the geodesic $\gamma(s) = \exp_p(s v/|v|)$  is the shortest path from $p$ to some point in $\partial C_{a}$.
\ep 

Further details about the structure of the sublevel sets of the Busemann function $b$ are given by
\begin{lem}[Guijarro-Kapovitch, cf.\cite{GK}] \lb{GuiKa}
Consider $C_\ell = b^{-1}\((-\infty, \ell]\) \subset M$ and $p \in \partial C_\ell$. Take $\gamma$ any minimal geodesic from $p= \gamma(0)$ to any point of the soul. Then there exists $\eps(\ell)$, with $\eps(\ell) \to 0$ as $\ell \to \infty$ such that if $v\in T_p M$ is a unit vector with $\angle(v, \dot{\gamma}(0)) < \frac{\pi}{2} - \eps(\ell)$, then $v \in T_p C_\ell$.
\end{lem}

The following theorem gives the existence of a {\bf tubular neighborhood} $U$:
\bt[Walter, cf.~\cite{Walter}] \lb{tubular} For each closed locally convex set $A \subset (M, g)$, there is an open set $U \subset A$ such that

1) For each $q \in U$, there is a unique $q^\ast\in A$ with $d(q, q^\ast) = d(q, A)$,
and a unique minimal geodesic from $q$ to $q^\ast$ which lies entirely in $U$.

2) $d_A$ is $C^1$ in $U \setminus A$ and twice differentiable almost everywhere in $U \setminus A$.

\et

Let us recall the {\bf Hessian bounds in the support sense}: 
\bde[Calabi, cf.~\cite{Cal}] \lb{def_spt}
Let $f: (M, g) \fle \re$ be continuous. We say that $\nabla^2 f|_p \geq h(p)$ in the support sense, for some function $h: M \fle \re$, if for every $\eps > 0$ there exists a smooth function $\varphi_\eps$ defined on a neighborhood of $p$ such that

1)  $\varphi_\eps (p) = f(p)$ and $\varphi_\eps \leq f$ in some neighborhood of $p$.

2) $\nabla^2 \varphi_\eps|_p \geq (h - \eps) g_p$.

\noindent Such functions $\varphi_\eps$ are called lower support functions of $f$ at $p$. One can analogously define $\nabla^2 f \leq h$ at $p$ in the support sense. 
\ede

\section{Miscellanea of Ricci flow results} \lb{AppC}

\subsection{Smooth convergence of manifolds and flows}

\bde[Cheeger-Gromov convergence] \lb{def_conv} (a) Consider a sequence of complete manifolds $(M^n_i, g_i)$ and choose $p_i \in M_i$. We say that $(M_i, g_i, p_i)$ converges to the pointed Riemannian $n$-manifold $(M_\infty, g_\infty, p_\infty)$ if there exists
\bi
\item[(1)] a collection $\{U_i\}_{i \geq 1}$ of compact sets with $U_i \subset U_{i + 1}$,  $\cup_{i\geq 1} U_i = M$ and  $p_\infty \in {\rm int}(U_i)$ for all $i$, and

\item[(2)] $\phi_i: U_i \fle M_i$  diffeomorphisms onto their image, with $\phi_i(p_\infty) = p_i$
\ei
such that $\phi^\ast_i g_i \to g_\infty$ smoothly on compact subsets of $M_\infty$, meaning that
$$|\nabla^m(\phi_i^\ast g_i - g_\infty)| \flecha 0 \qquad \text{as} \quad i \to \infty \quad \text{ on } \  K \quad \text{for all}\quad m \geq 0$$
for every compact set $K \subset M$. Here $|\cdot|$ and $\nabla$ are computed with respect to any fixed background metric.

(b) A sequence of complete evolving manifolds $(M_i, g_i(t), p_i)_{t \in I}$ converges to a pointed evolving manifold $(M_\infty, g_\infty(t), p_\infty)_{t \in I}$ if we have (1) and (2) as before such that $\phi^\ast_i g_i(t) \to g_\infty(t)$ smoothly on compact subsets of $M_\infty \times I$.
\ede

\bt[Hamilton, cf.~\cite{Hamform}] \lb{locHCT}
Let $(M_k, g_k(t), x_k)_{t \in (a, b]}$ be complete $n$-dimen\-sional Ricci flows, and fix $t_0 \in (a, b]$. Assume the following two conditions:
\bi
\item[(1)] For each compact interval $I \subset (a, b]$, there is a constant $C = C(I) < \infty$ so that for all $t \in I$ 
$$|\Rm|_{g_k(t)} \leq C \qquad \text{on} \qquad B_{g_k(0)}(x_k, r)\hspace*{1em} \mbox{ for all \quad $k\ge k_0(r)$.}$$

\item[(2)] There exists $\delta > 0$ such that $inj_{g_k(t_0)} (x_k) \geq \delta.$
\ei
Then, after passing to a subsequence, the solutions converge smoothly to a complete Ricci flow solution $(M_\infty, g_\infty(t), x_\infty)$ of the same dimension, defined on $(a, b]$. 
\et

Some authors quote stronger versions of this theorem,  
where the curvature bound $C$ is allowed to increase arbitrarily with 
$r$. However, in the proof one then  runs into trouble if one wants to verify completeness 
of the limit metrics for different times. 
Lemma~\ref{lem: complete2} can be regarded as a 
way to circumvent this problem.

Under  bounded curvature, condition (2)  above can be guaranteed by ensuring a lower bound on the volume (see  \cite[Theorem 4.3]{ChGrT}):
\bt[Cheeger-Gromov-Taylor] \lb{ChGrT} Let $B_g(p, r)$ be a metric ball in a complete Riemannian manifold $(M^n, g)$ with $\lambda \leq K_g|_{B_g(p, r)} \leq \Lambda$ for some constants $\lambda, \Lambda$. Then, for any constant $r_0$ such that $4 r_0 < \min\{\pi/\sqrt{\Lambda}, r\}$ if $\Lambda > 0$, we have
$${\rm inj}_g(p) \geq r_0 \(1 + \frac{V_\lambda^n(2 r_0)}{\vle_g \(B_g(p, r_0)\)}\)^{-1},$$
where $V_\lambda^n(\rho)$ denotes the volume of the ball of radius $\rho$ in the $n$-dimensional space with constant sectional curvature $\lambda$.
\et

\subsection{Curvature estimates}

Shi's local derivative estimates ensure that if the curvature is bounded on $B_{g(0)}(p, r) \times [0, T]$, then we also have bounds on all covariant derivatives of the curvature on the smaller set $B_{g(0)}(p, r/2) \times (0, T]$, where such bounds blow up to infinity as $t \to 0$. Such a degeneracy can be avoided by making the stronger assumption of having bounded derivatives of the curvature in the initial metric.
\bt[Lu-Tian, \cite{LT}] \lb{ModShi} For any positive numbers $\alpha, K, K_\ell, r, n \geq 2, m \in \mathbb N$, let $M^n$ be a manifold with $p \in M$, and $g(t)$, $t\in [0, \tau]$ where $\tau \in (0, \alpha/K)$, be a Ricci flow on an open neighborhood $\mc U$ of $p$ containing $\overline B_{g(0)}(p, r)$ as a compact subset. If
$$|\Rm_{g(t)}|(x) \leq K \qquad \text{for all} \ x \in B_{g(0)}(p, r) \quad \text{and} \quad t\in [0, \tau]$$
$$|\nabla^\ell \Rm_{g(0)}|(x) \leq K(\ell) \qquad \text{for all} \ x \in B_{g(0)}(p, r) \quad \text{and all} \quad \ell \geq 0,$$
then there exists $C = C(\alpha, K, K(\ell), r, m, n)$ such that
$$|\nabla^m \Rm_{g(t)}| \leq C \qquad \text{on} \quad \overline B_{g(0)}(p, r/2) \times [0, \tau].$$

\et
Next we state a result 
of Simon \cite[Theorem 1.3]{Simon0}. We actually 
use  a simplified and coordinate free 
version, see also
Chen \cite[Corollary 3.2]{strong}:
\bt[M.~Simon, B.~L.~Chen] \lb{SURF}
Let $(M^n, g(t))$, with $t\in [0, T]$, be a complete Ricci flow. Assume we have the curvature bounds
\bec \lb{R_bdd_0}
|\Rm|_{g(0)} \leq \rho^{-2}  \qquad \text{on } \quad B_{g(0)}(p, \rho)
\eec
and
\bec \lb{R_bdd_t}
|\Rm|_{g(t)}(x) \leq K/t  \qquad \text{for } \quad x \in B_{g(0)} (p, \rho) \quad \text{ and } \quad t\in (0, T].
\eec
Then there exists a constant $C$ depending only on $n$
such that
$$|\Rm|_{g(t)}(x) \leq 4\, e^{CK} \rho^{-2} \qquad \text{ for all} \qquad x \in B_{g(0)} (p, \rho/2), \quad  t\in [0, T].$$
\et

\begin{footnotesize}
\hspace*{0.3em}{\sc University of M\"unster,
Einsteinstrasse 62, 48149 M\"unster, Germany}\\
\hspace*{0.3em}{$\begin{array}{l}\mbox{\em E-mail addresses:}\\\hspace*{1em}\end{array} $}{$\begin{array}{l}\mbox{\sf E.Cabezas-Rivas@uni-muenster.de,}\\ \mbox{\sf wilking@math.uni-muenster.de}\end{array}$}
\end{footnotesize}

\end{document}